%% file: main.tex
\documentclass[12pt,oneside,reqno]{amsart}

\usepackage{point-spectrum}

\newcommand{\jacobi}[1]{A_{#1}}
\newcommand{\aomoto}[2]{X_{#1}(#2)}
\newcommand{\aom}[2]{X_{#1}(#2)}
\newcommand{\T}{\mathcal{T}}
\newcommand{\F}{\mathcal{F}}
\renewcommand{\P}{\mathcal{P}}

\newcommand{\ind}[2]{I_{#1}(#2)}
\newcommand{\cc}{\mathsf{cc}}
\newcommand{\from}{\leftarrow}

\usepackage[scr=boondoxo]{mathalfa}
\renewcommand{\a}{\mathscr{a}}
\renewcommand{\b}{\mathscr{b}}

\title{Point Spectrum of Periodic Operators on Universal Covering Trees}

\usepackage[foot]{amsaddr}

\makeatletter
\renewcommand{\email}[2][]{%
  \ifx\emails\@empty\relax\else{\g@addto@macro\emails{,\space}}\fi%
  \@ifnotempty{#1}{\g@addto@macro\emails{\textrm{(#1)}\space}}%
  \g@addto@macro\emails{#2}%
}
\makeatother

\author{Jess Banks}
\author{Jorge Garza-Vargas}
\author{Satyaki Mukherjee}
\address{UC Berkeley}
\email{jess.m.banks@berkeley.edu, jgarzavargas@berkeley.edu, satyaki@berkeley.edu}

\begin{document}
\maketitle
\begin{abstract}
    For any multi-graph $G$ with edge weights and vertex potential, and its universal covering tree $\T$, we completely characterize the point spectrum of operators $\jacobi{\T}$ on $\T$ arising as pull-backs of local, self-adjoint operators $\jacobi{G}$ on $G$. This builds on work of Aomoto, and includes an alternative proof of the necessary condition for point spectrum derived in \cite{aomoto1991point}. Our result gives a finite time algorithm to compute the point spectrum of $\jacobi{\T}$ from the graph $G$, and additionally allows us to show that this point spectrum is contained in the spectrum of $\jacobi{G}$. Finally, we prove that typical pull-back operators have a spectral delocalization property: the set of edge weight and vertex potential parameters of $\jacobi{G}$ giving rise to $\jacobi{\T}$ with purely absolutely continuous spectrum is open and its complement has large codimension.
\end{abstract}

\input{Content/intro.tex}

\input{Content/prelims.tex}

\input{Content/main-results.tex}

\input{Content/acyclictree.tex}

\input{Content/mass_formula}

\input{Content/grapheigenvalues}

\input{Content/Delocalization_new}

\bigskip
\subsection*{Acknowledgements}
We are grateful to Nikhil Srivastava for many helpful discussions. We also thank Barry Simon for his many helpful comments and suggestions on a previous version of this manuscript, and in particular for his guidance in showing the current version of Theorem \ref{thm:setofparameterswithpp}. JGV also thanks Irit Huq-Kuruvilla for patiently clarifying many basic concepts in geometry.   JB is supported by the NSF Graduate Research Fellowship Program under Grant DGE-1752814. JGV and SM are supported by NSF Grants CCF-1553751 and CCF-2009011.

\bibliographystyle{alpha}
\bibliography{point-spectrum}

\end{document}

%% file: Content/intro.tex
\section{Introduction}

Consider a finite graph $G=(V, E)$ and its universal cover $\T=(\mathcal{V}, \mathcal{E})$,  together with a covering map $\Xi: \T \to G$. The purpose of this paper is to relate the point spectrum of certain local, periodic,   self-adjoint operators on $\ell^2(\calV)$ to the combinatorial structure of $G$.

Precise definitions, notation and assumptions about the model in consideration will be discussed below in Section \ref{sec:prelims}, but for now we give a high-level overview of the problem setting. By endowing $G$ with edge weights and a potential on its vertex set, we obtain a natural self-adjoint operator $\jacobi{G}$ on $\ell^2(V)$. This framework encompasses Schr\"odinger operators on graphs, weighted adjacency matrices, graph Laplacians and transition matrices for random walks, and the corresponding pull-back of the weights and potential to $\T$ induces an analogous periodic, self-adjoint operator $\jacobi{\T}$ on $\ell^2(\calV)$.

The class of operators $A_\T$ obtained in this way contains, but is richer than, the periodic Schr\"{o}dinger operators in one dimension, which are of great relevance to spectral theory and the theory of orthogonal polynomials. The spectra of these $A_\T$ are additionally crucial to the study of relative expanders \cite{friedman2003relative} and, as shown in \cite{bordenave2019eigenvalues}, control in a strong sense the spectra of large random lifts of a fixed base graph. However, despite the many advances in functional analysis, operator algebras and operator theory, many natural questions regarding the spectral properties of $A_\T$ are unanswered and seem inaccessible with current techniques. We direct the reader to \cite{avni2020periodic} for a survey of both periodic Jacobi matrices and the difficulty in generalizing to the more generic case considered here.

In this paper we will be concerned with the spectrum of $\jacobi{\T}$, which we denote by $\Spec \jacobi{\T}$, the density of states $\mu$ (a natural and canonical measure on $\Spec \jacobi{\T}$), and most importantly those $\lambda \in \Spec \jacobi{\T}$ for which there exists a corresponding $\ell^2$ eigenvector---in other words, the point spectrum $\Spec_p \jacobi{\T}$.

Our main result is a set of necessary and sufficient condition on $G$ (including its edge weights and potential) for $\Spec_p \jacobi{\T}$ to be non-empty. This gives a finite algorithm to compute $\Spec_p \jacobi{\T}$ from $G$, and extends the work of Aomoto, who has already shown the necessary half of our result in \cite{aomoto1991point}. However, our new and elementary argument is essentially different from his, and we build on it to show, surprisingly, that $\Spec_p \jacobi{\T} \subset \Spec \jacobi{G}$, and to give a lower bound for the multiplicity of each eigenvalue  of $\jacobi{G}$ arising in this way. Finally, we prove that the set of edge weights and potentials for which $\jacobi{\T}$ has point spectrum is a closed semialgebraic set of large codimension, which in particular implies that the set has Lebesgue measure zero. This may be regarded as a spectral delocalization result of the kind long-studied in mathematical physics \cite{anantharaman2018delocalization}; see \cite{anantharaman2020absolutely} for recent and analogous work in the context of quantum graphs. In particular, our result implies that even when $  \Spec_p \jacobi{\T}$ has an isolated point, there are arbitrarily small perturbations of $\jacobi{\T}$ with no point spectrum at all. In view of the Kato-Rellich theorem on stability of the discrete point spectrum, this is a strong manifestation of the fact that the eigenspaces of $A_\T$ are infinite-dimensional. 

\subsection*{Related Work}

The operators $\jacobi{\T}$ defined here have been studied by several authors with different motivations and levels of generality, and are variously referred to as \textit{operators of nearest-neighbor type} \cite{aomoto1991point},  \textit{connected, local, pull-back operators} \cite{angel2015non} or \textit{periodic Jacobi operators} \cite{avni2020periodic}; we will use the latter. When $G$ is an unweighted $d$-regular graph (making $A_G$ is its adjacency matrix), classical work of Kesten in the context of Cayley graphs \cite{kesten1959symmetric}, and McKay in the context of random graphs \cite{mckay1981expected}, proved that $\Spec \jacobi{\T} = [-2\sqrt{d-1}, 2\sqrt{d-1}]$ and that $\mu$ follows what is now called the Kesten-Mackay distribution with parameter $d$. When $G$ is an unweighted $(c, d)$-bireguar bipartite graph with $c < d$, Godsil and Mohar showed that $\Spec \jacobi{\T} = \{\lambda \in \bbR: \sqrt{d-1}-\sqrt{c-1} \leq |\lambda|\leq \sqrt{c-1}+\sqrt{d-1} \}\cup \{0\}$ and that $\mu\{0\} = \frac{d-c}{d+c}$ \cite{godsil1988walk}. These results imply that for adjacency matrices, when $G$ is  $d$-regular, $\jacobi{\T}$ has no point spectrum, while when $G$ is $(c, d)$-biregular and bipartite,  $\Spec_p \jacobi{\T} =\{0\}$. 

Subsequent work focused on the properties of $\Spec \jacobi{\T}$ and $\mu$, and their relation to $\jacobi{G}$, without making any assumptions on $G$; see \cite{aomoto1988algebraic, aomoto1991point, sunada1992group, sy1992discrete} as well the more recent \cite{angel2015non, bordenave2019eigenvalues, avni2020periodic, vargas2019spectra}. Of relevance for the current paper is a result of  Avni, Breuer and Simon in \cite{avni2020periodic}, which states that for any $G$, any edge weights, and any potential, the operator $\jacobi{\T}$ has no singular continuous spectrum. As a corollary one can deduce that the continuous part of $\Spec \jacobi{\T}$ always consists of a finite union of closed non-degenerate intervals, and its singular part is the finite set of eigenvalues of $\Spec_p \jacobi{\T}$. Equivalently, $\mu$ can be decomposed into a measure that is absolutely continuous with respect to the Lebesgue measure on $\bbR$ and a finite sum of atomic measures. 

The most noteworthy prior result regarding the point spectrum of $\jacobi{\T}$ is the aforementioned work of Aomoto, who in addition to deriving necessary conditions for the presence of point spectrum of $\jacobi{\T}$ deduced a remarkable formula relating $\mu\{\lambda\}$ to the combinatorial structure of $G$ for every $\lambda \in \Spec_p \jacobi{\T}$. He then used these results to show that when $G$ is a $d$-regular graph, regardless of the edge weights and potential, $\jacobi{\T}$ has no point spectrum. This generalizes the case when $G$ is a cycle, which was established by different authors in the mathematical physics literature;  see Section 2 of \cite{avni2020periodic} for a discussion and survey. In a different context, Keller, Lenz and Warzel \cite{keller2013spectral} showed that adjacency matrices of certain trees have no point spectrum and that this property is stable under small perturbations of the potential. For our setting, their results imply that if $G$ has a loop at every vertex and $A_G$ is the adjacency matrix of $G$, then $A_\T$ has no point spectrum. 

Our results, stated in Section \ref{sec:mainresults} after the preliminary material in Section \ref{sec:prelims}, recover many of the ones above and provide a pleasant unification and generalization of the literature on point spectra.

%% file: Content/prelims.tex
\section{Preliminaries}
\label{sec:prelims}

\subsection{Graphs and Covers}
\label{subsec:graphsandcovers}

We will work in the general setting of weighted graphs with self-loops and multi-edges. In this setup we will regard a graph as a tuple $G = (V,E,a,b)$, consisting of vertices, edges, edge-weights $a : E \to \bbC$ and a potential $b : V \to \bbC$. When it is not clear from the context, we will write $V(G)$ and $E(G)$ to emphasize that we are referring to the set of vertices and edges of $G$.  It will be convenient to regard $E$ as a set of directed edges, equipped with a direction-reversing involution $e \mapsto \check e$ with no fixed points, as well as \textit{source} and \textit{terminal} maps $\sigma,\tau : E \to V$ so that $\sigma(e) = \tau(\check e)$ for every $e \in E$. An edge for which $\sigma(e) = \tau(e)$ and $e = \check e$ is a \textit{self-loop}, and we refer to the remainder as \textit{proper edges}.%
\footnote{
    Some authors additionally include so-called \textit{half-loops}, which are edges $e$ with $\sigma(e) = \tau(e)$ and $e = \check e$; see \cite{friedman1993}. Our results easily extend to this case, but for simplicity we will not consider it here.
}
We will also abuse notation and write $\sigma(u)$ and $\tau(u)$ for the sets of directed edges whose source and terminal, respectively, are the vertex $u \in V$. 

We say that a graph $H$ \textit{covers} $G$ if there exists a \textit{covering map} $\xi : H \to G$, namely a map of vertices and edges which is compatible with the source and terminal maps, preserves potential and edge weights, and is an isomorphism on $\sigma(u)$ and $\tau(u)$ for each vertex $u$. If both are finite, then $|V(G)|$ necessarily divides $|V(H)|$, and we call their ratio $n$ the degree of the cover; equivalently we say that $H$ is an 
\textit{$n$-lift} of $G$. Each $n$-lift $H$ may be expressed explicitly by an assignment of permutations to edges $\pi : E \to \frS_n$, with the property that $\pi_e^{-1} = \pi_{\check e}$ for each edge $e \in E$. Then $V(H) = V(G) \times [n]$---throughout the paper we will use the notation $[n] = \{1,...,n\}$---and for every $e \in E(G)$ and each $i \in [n]$, we include an edge $\tilde e \in E(H)$ with $\sigma(\tilde e) = (\sigma(e),i)$ and $\tau(\tilde e) = (\tau(e),\pi_e(i))$.

The \textit{universal cover} of a connected graph $G$ is the unique (up to isomorphism) infinite tree $\calT = (\calV,\calE,\a,\b)$ that covers every other cover of $G$. It can be constructed directly in terms of \textit{non-backtracking walks} on $G$, which are sequences of edges $e_1,e_2,...e_\ell$ such that, for every $s \in [\ell - 1]$, $\tau(e_s) = \sigma(e_{s+1})$ and $e_s \neq \check e_{s+1}$. If we choose a root vertex $u \in V$, then we may set the vertex set $\calV$ of $\calT$ to be the set of non-backtracking walks on $G$ starting at $u$, with directed edges $\calE$ whenever one walk is an immediate prefix or suffix of another, and edge weights and potential inherited from the final edge and vertex of the walk, respectively. Up to isomorphism $\calT$ is independent of the root choice, and is manifestly a cover of $G$; we will call the covering map $\Xi$. Finally, we note that $\calT$ is finite if and only if $G$ is \textit{acyclic}---that is, if it does not contain a closed non-backtracking walk. In this case $G = \calT$.

Given $G$ and a universal cover $\calT$, the latter is endowed with a set of symmetries which act transitively on $\calV$ by simultaneously permuting the fibres over every vertex.

\subsection{Jacobi Operators, Spectra, and the Density of States}
\label{subsec:jacobiops}
Following the convention introduced in \cite{avni2020periodic}, the \textit{Jacobi operator} associated to $G = (V,E,a,b)$ acts on $\eta \in \ell^2(V) \simeq \bbC^{|V|}$ as:
\begin{equation}
    \label{eq:jacobi-def}
    (\jacobi{G} \eta)(u) = b_u \eta(u) + \sum_{e \in \tau(u)} a_e \eta(\sigma(e)).
\end{equation}
Throughout the paper, we will assume that the edge weights satisfy a conjugate symmetry condition $\conj{a_e} = a_{\check e}$ and that the potential $b$ is real---these ensure that $\jacobi{G}$ is Hermitian, and we will accordingly call such edge weights and potential Hermitian as well. 

When $H$ is an $n$-lift of $G$, we will always think of $A_H$ as acting on $\ell^2(V)\otimes \bbC$, regarded as the set of $\bbC^n$-valued functions on $V$. A standard result characterizes the spectrum of $\jacobi{H}$. Let $\pi : E \to \frS_n$ be the set of permutations which define $H$, and overload notation to write $\pi_e$ as well for the unitary operator which acts by permuting the coordinates of $\bbC^n$ according to $\pi_e$. Then $\jacobi{H}$ acts on $\eta \in \ell^2(V) \otimes\bbC^n$ as
\begin{equation*}
    (\jacobi{H} \eta)(u) = b_u \eta(u) + \sum_{e \in \tau(u)} a_e \pi_e \eta(\sigma(e)) \in \bbC^n.
\end{equation*}

By writing $\bbC^n = \bbC^n_1 \oplus \bbC^n_0 \simeq \bbC \oplus \bbC^{n-1}$, where $\bbC^n_1$ is the span of the all-ones vector and $\bbC^n_0$ is its orthogonal complement, we can simultaneously decompose every edge permutation as $\pi_e = 1 \oplus \rho(\pi_e)$; the latter unitary operator on $\bbC^n_0 \simeq \bbC^{n-1}$ is the \textit{regular representation} of $\pi_e$. Thus we may write
\begin{equation}
    \jacobi{H} = \jacobi{G} \oplus \jacobi{H/G},
\end{equation}
where the second acts on $\eta \in \ell^2(V) \otimes \bbC^n_0$ in the natural way:
\begin{equation}
    (\jacobi{H/G} \eta)(u) = b_u \eta(u) + \sum_{e \in \tau(u)} a_e \rho(\pi_e) \eta(\sigma(e)) \in \bbC^{n}_0.
\end{equation}
In other words,
\begin{equation}
    \Spec \jacobi{H} = \Spec\jacobi{G} \sqcup \Spec\jacobi{H/G},
\end{equation}
and we refer to $\Spec \jacobi{H/G}$ as the \textit{new eigenvalues} of $H$.

Once again writing $\T = (\calV,\calE,\a,\b)$ for the universal cover of $G$, we will call the analogous operator on $\ell^2(\calV)$ the \textit{periodic Jacobi operator} of $\T$. Since the edge weights $\a$ and potential $\b$ are related to those of $G$ by $\a_{\tilde e} = a_{\Xi(\tilde e)}$ for every $\tilde e \in \calE$ and $\b_{\tilde v} = b_{\Xi(\tilde v)}$ for every $v \in \calV$, finiteness of $G$, $a$, and $b$ ensure that $\jacobi{\T}$ belongs to the set $\calB(\ell^2(\calV))$ of bounded operators on $\ell^2(\calV)$, and the inherited conjugate symmetry condition $\conj{\a_e} = \a_{\check e}$ guarantees that it is Hermitian. We use $\Spec A_\T$ to denote the \textit{spectrum} of the periodic Jacobi operator, that is
\begin{equation}
    \Spec \jacobi{\calT} = \left\{ \lambda \in \bbR: (\lambda - \jacobi{\calT})^{-1} \notin \calB(\ell^2(\calV)) \right\}.
\end{equation}
We remind the reader that, unlike in the finite dimensional case, $\lambda \in \Spec \jacobi{\calT}$ does not guarantee an $\ell^2$ eigenvector for $\lambda$. The subset of the spectrum with this additional property---the \textit{point spectrum}---will be our primary concern in this work. We will return to it below.

For any $u \in V(G)$, the quantities $\langle \delta_{\tilde u},  \jacobi{\calT}^\ell \delta_{\tilde u} \rangle$ for $\ell \in \bbN$ are real and constant over all $\tilde u$ in the fibre over $u$, and a routine application of the Riesz representation theorem guarantees an accompanying \textit{spectral measure} $\mu_u$ on $\Spec \jacobi{\calT}$ associated to each $u$, satisfying
\begin{align}
    \langle \delta_{\tilde u},  f(\jacobi{\calT}) \delta_{\tilde u}\rangle  = \int_{\Spec \jacobi{\calT}} f(x) \dee\mu_u(x) & & \forall \tilde u \in \Xi^{-1}(u)
\end{align}
for every bounded measurable function $f : \Spec \jacobi{\calT} \to \bbC$, where $f(A_\T)$ is defined via the Borel functional calculus. The \textit{density of states (DOS)} of $\jacobi{\calT}$ is the unique measure obtained by averaging these spectral measures over $u \in V(G)$:
\begin{equation}
    \mu = \frac{1}{|V(G)|}\sum_{u \in V(G)}\mu_u.
\end{equation}

It is typical in the literature to work with real positive edge weights instead of Hermitian ones. The latter choice will make some of our proofs more convenient, but from the perspective of $\Spec A_{\T}$ it does not add any generality. In particular, $A_{\T}$ is \textit{gauge invariant} in the following sense.

\begin{lemma}
    \label{lem:gauge}
    Let $G = (V,E,a,b)$ be a graph with Hermitian edge weights and real potential, and let $G' = (V,E,a',b)$, where $a'_e = |a_e|$; write $\T$ and $\T'$ for their respective covers. Then $A_\T = U^\ast A_{\T'} U$, where $U$ is a diagonal unitary operator.
\end{lemma}

\noindent One may prove Lemma \ref{lem:gauge} by choosing a root $r$ for $\T$ and letting $U_{v,v}$ be the product of edge weight arguments along the unique shortest path connecting $r$ and $v$. An immediate corollary is that \textit{the spectrum and density of states of $\jacobi{\T}$ depend only on the moduli of the edge weights}. In the case $b \equiv 0$, note that the above implies that $A_{\mathcal{T}}$ and $-A_\mathcal{T}$ are unitarily equivalent,  which has the following consequence.

\begin{lemma}
\label{lem:symmetryofspectrum}
     Let $G$ be a finite graph, $\T$ its universal cover. If $b_v=0$ for all $v\in V(G)$ then the spectrum and density of states of $\jacobi{\T}$ are symmetric about zero. 
\end{lemma}


On several occasions we will use the following well-known facts to relate the empirical spectral measures of finite graphs $G$ to the densities of states of their universal covers. 

\begin{lemma}
    \label{lem:lifts-weak-conv}
    Let $G$ be a finite graph, $\T$ its universal cover, and $\mu$ the density of states of $\jacobi{\T}$. There exists a sequence of covers $G_n$ of $G$ whose girth%
    \footnote{The girth of a graph is the length of its shortest cycle.} 
    diverges as $n$ goes to infinity.  Moreover, for this sequence, the empirical spectral measures $\mu_n$ of $\jacobi{G_n}$ converges weakly to $\mu$.
\end{lemma}

This lemma  follows directly  from results stated in \cite{avni2020periodic} whose proofs will appear in \cite{boundaryconditions2020}. Here we discuss very briefly their approach and refer the reader to Section 4 of \cite{avni2020periodic} for details. Begin by fixing a spanning tree $T$ of $G$ and define $l = |E(G)\setminus E(T)|$. Then, if $G$ is viewed as a 1-complex, the fundamental group of $G$ is the free group on $l$ generators, namely, $\textbf{F}_l$. Moreover, by contracting the lifts of $T$ in $\T$ one obtains the Cayley graph of $\textbf{F}_l$ and this allows to define an action of $\textbf{F}_l$ on $\T$ which has the copies of $T$ as fundamental domains. Subgroups of $\textbf{F}_l$ of finite index correspond to the lifts of $G$, and those lifts coming from normal subgroups enjoy a certain type of symmetry. Theorem 4.3 in \cite{avni2020periodic} says that one can take a sequence of normal subgroups of $\textbf{F}_l$ of finite index, which give rise to a sequence of lifts of $G$, say $G_1, G_2, \dots$, with the property that the empirical spectral distributions of $A_{G_n}$ converges weakly to $\mu$. It is not hard to see from their discussion that the symmetric property of the $G_n$ together with the weak convergence of the spectral distributions,  implies that this sequence of graphs has diverging girth.  

For completeness, below we provide a purely combinatorial proof of Lemma \ref{lem:lifts-weak-conv}. 

\begin{proof}[Proof of Lemma \ref{lem:lifts-weak-conv}]
    By induction, it suffices to show that for every finite graph $H = (V,E)$ with $\girth (H) = p$ and $|E| = 2m$, there exists a finite lift $L$ of $H$ whose girth is strictly larger (the weights and potential are irrelevant here, and we will suppress them). We will construct $L$ as a $2^{m+1}$-lift of $H$, with the following set of permutations $\pi : E \to \frS_{2^{m} + 1}$. Group the edges in to pairs $(e,\check e)$ consisting of an edge and its reversal, and order these $(e_1,\check e_1),...,(e_m, \check e_m)$. Now let $\pi_{e_i}$ be the permutation that maps $j \mapsto j + 2^i \mod 2^{m + 1}$ for every $j \in [2^{m + 1}]$, and let $\pi_{\check e_i} = \pi_{e_i}^{-1}$ as required.
    
    Since $\girth L > \girth H$, we need only to show that $L$ contains no cycle of length $p$. Seeking contradiction, assume instead that $e_{i_1},...,e_{i_p}$ is a sequence of $p$ directed edges forming a cycle in $L$. Writing $\xi$ for the covering map, and $\xi(e_{i_1}),...,\xi(e_{i_p})$ form a cycle in $H$ with length $p$, and since $\girth H = p$, they are distinct. The vertices of $L$ are $V \times [2^{m+1}]$, which we regard as a set of $2^{m+1}$ `layers;' assume $\sigma(e_{i_1})$ is in the $t$th one. Because of how we have arranged the permutations, $\tau(e_{i_p})$ is in layer $t \pm 2^{i_1} \pm \cdots \pm 2^{i_p} \neq t \mod 2^{m+1}$, because the $i_1,...,i_p$ are distinct and smaller than $m+1$. Thus $\tau(e_{i_p}) \neq \sigma(e_{i_1})$---a contradiction.

    We finally show that, given such a sequence $G_n$ with diverging girth, $\mu_n$ converges weakly to $\mu$. For every fixed positive integer $k$ and each vertex $u$ of $G_n$, the quantity $\langle \delta_u, \jacobi{G_n}^k \delta_u\rangle$ is a weighted count of length-$k$ closed walks in $G_n$ starting and ending at $u$. Since the $G_n$ have diverging girth, for $n$ sufficiently large the depth-$k$ neighborhood of $u$  in $G_n$ is identical to that of every $\tilde u \in \Xi^{-1}(u)$, and thus this count is eventually constant and equal to $\langle \delta_{\tilde u}, \jacobi{\T}^k \delta_{\tilde u} \rangle$. Finally, as the $k$th moments of the empirical spectral measures $\mu_n$ are given by normalized traces of $\jacobi{G_n}^k$, the method of moments gives us weak convergence to the density of states.
\end{proof}

Substantially stronger versions of this result are known but will not be necessary for us; we direct the reader for instance to the recent work of Bordenave and Collins \cite{bordenave2019eigenvalues}.

\subsection{Point Spectrum and the Aomoto Sets}
\label{subsec:aomotosets}

We will denote the point spectrum of $\jacobi{\calT}$ as
\begin{equation}
    \Spec_p \jacobi{\calT} = \left\{ \lambda \in \bbR :  \Ker (\lambda - \jacobi{\calT}) \neq \{0\} \right\}.
\end{equation}
The following proposition collates several equivalent characterizations of $\Spec_p \jacobi{\calT}$. 

\begin{proposition}
\label{prop:equivallencesforeigs}
    Let $G$ be a finite graph. Assume $G$ has  at least one cycle, and let $\calT$ be its universal cover. Then $\lambda \in \Spec_p \jacobi{\calT}$ if and only if any of the following hold:
    \begin{enumerate}[(i)]
        \item \label{prop:equivallencesforeigs-i}  $\dim \Ker(\lambda - \jacobi{\T}) = \infty$
        \item \label{prop:equivallencesforeigs-ii} $\lambda$ is an atom of $\mu$
        \item \label{prop:equivallencesforeigs-iii} For some $u \in V(G)$, $\lambda$ is an atom of $\mu_u$.
        \item \label{prop:equivallencesforeigs-iv} For some $u \in V(G)$, the Cauchy transform
        $$
            S_u(z) = \int_{\Spec \jacobi{\calT}} (z - x)^{-1} \dee \mu_u(x)
        $$
        has a pole at $\lambda$.
        \item \label{prop:equivallencesforeigs-v} For some $u \in V(G)$, and every $\tilde u \in \Xi^{-1}(u)$, there exists $\zeta \in \Ker(\lambda - \jacobi{\calT})$ with $\zeta(\tilde u) \neq 0$.
    \end{enumerate}
    Moreover, the vertices satisfying (\ref{prop:equivallencesforeigs-iii}),(\ref{prop:equivallencesforeigs-iv}), and (\ref{prop:equivallencesforeigs-v}) coincide.
\end{proposition}

\begin{proof}
In Section 8 of \cite{avni2020periodic} it is proven that $\lambda \in \Spec_p A_\T$ implies \eqref{prop:equivallencesforeigs-i}.  The intuition behind this result is that if $\xi\in \Ker(\lambda -A_\T)$, then precomposing $\xi$ with deck transformations on $\T$ gives rise to many linearly independent eigenvectors of $A_\T$ for $\lambda$.  

The rest of the claims are standard results that are true for general self-adjoint bounded operators. The fact that $\lambda\in \Spec_p A_\T$, \eqref{prop:equivallencesforeigs-ii}, \eqref{prop:equivallencesforeigs-iii}, and \eqref{prop:equivallencesforeigs-v} are equivalent follows directly from the definition of $\mu$ and the forthcoming Lemma \ref{lem:atom-basis}. On the other hand \eqref{prop:equivallencesforeigs-iii} and \eqref{prop:equivallencesforeigs-iv} are clearly equivalent. 
\end{proof}

\noindent  By way of a complicated set of coupled equations satisfied by the Cauchy transforms $S_u$, Aomoto identified a set of vertices of $G$ whose combinatorial structure is instrumental in understanding $\Spec_p A_\T $ and will be the focus of much of this paper.

\begin{definition}[The Aomoto set]
Let $G$ be a finite graph and assume that $\lambda \in \Spec_p A_\T $. The \textit{Aomoto set} of $G$ associated to $\lambda$ consists of those vertices in $V(G)$ that satisfy the equivalent conditions (\ref{prop:equivallencesforeigs-iii}-\ref{prop:equivallencesforeigs-v}) in Proposition \ref{prop:equivallencesforeigs}. This set will be denoted by $X_\lambda(G)$.\footnote{This set was referred to as $X_{\lambda}^{(1)}(G)$ in \cite{aomoto1991point} and \cite{avni2020periodic}; we have dropped the superscript to lighten notation, and because we will not consider the sets $X_{\lambda}^{(\alpha)}(G)$ for $\alpha \neq 1$ which appear in that work.} 
\end{definition}

\begin{figure}[h]
   
\centering
     \begin{tabular}{c c c}
        \includegraphics[scale=.75]{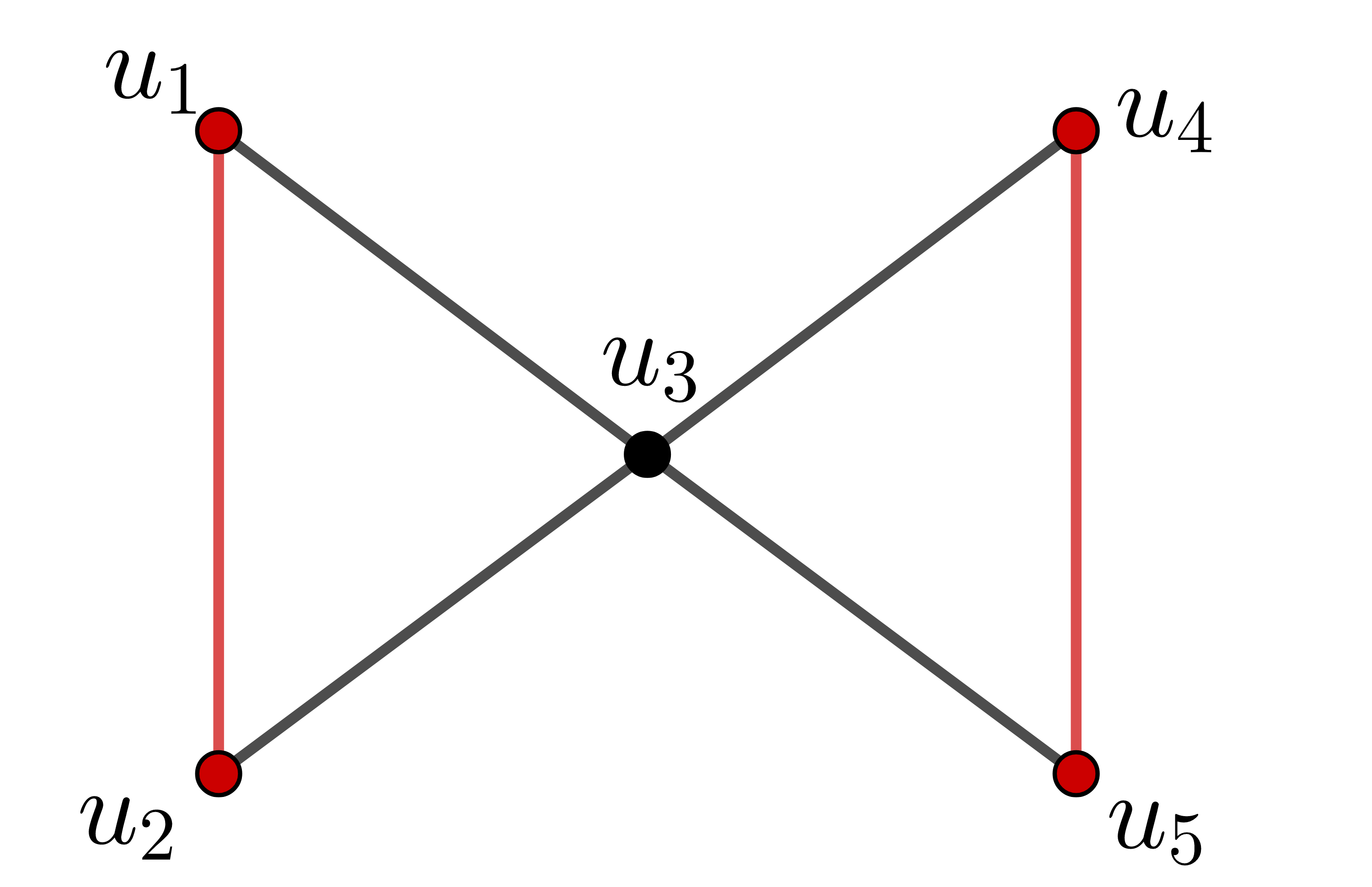} \quad &
         \includegraphics[scale =.195]{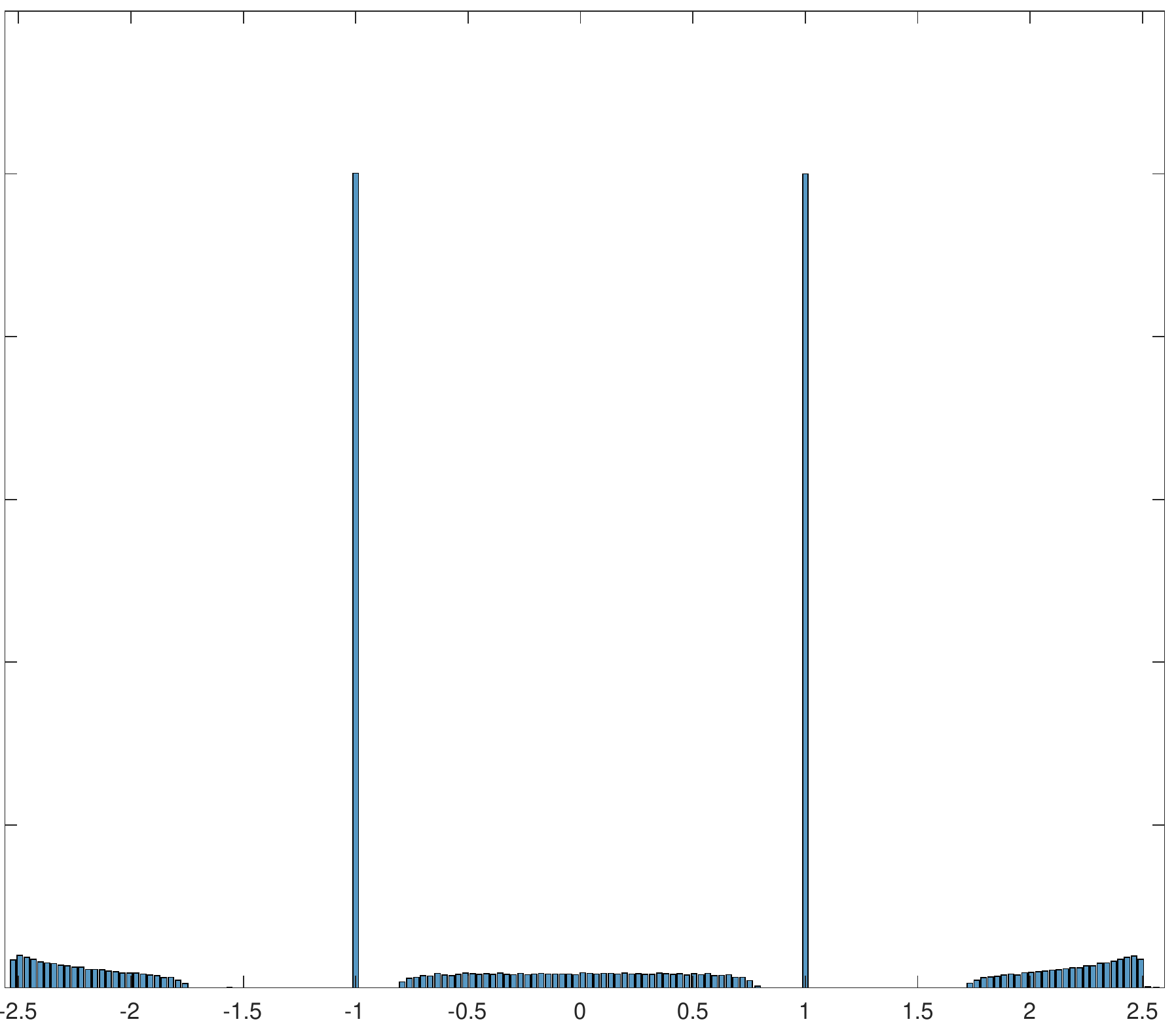} \hspace{.6cm} & \includegraphics[scale =.195]{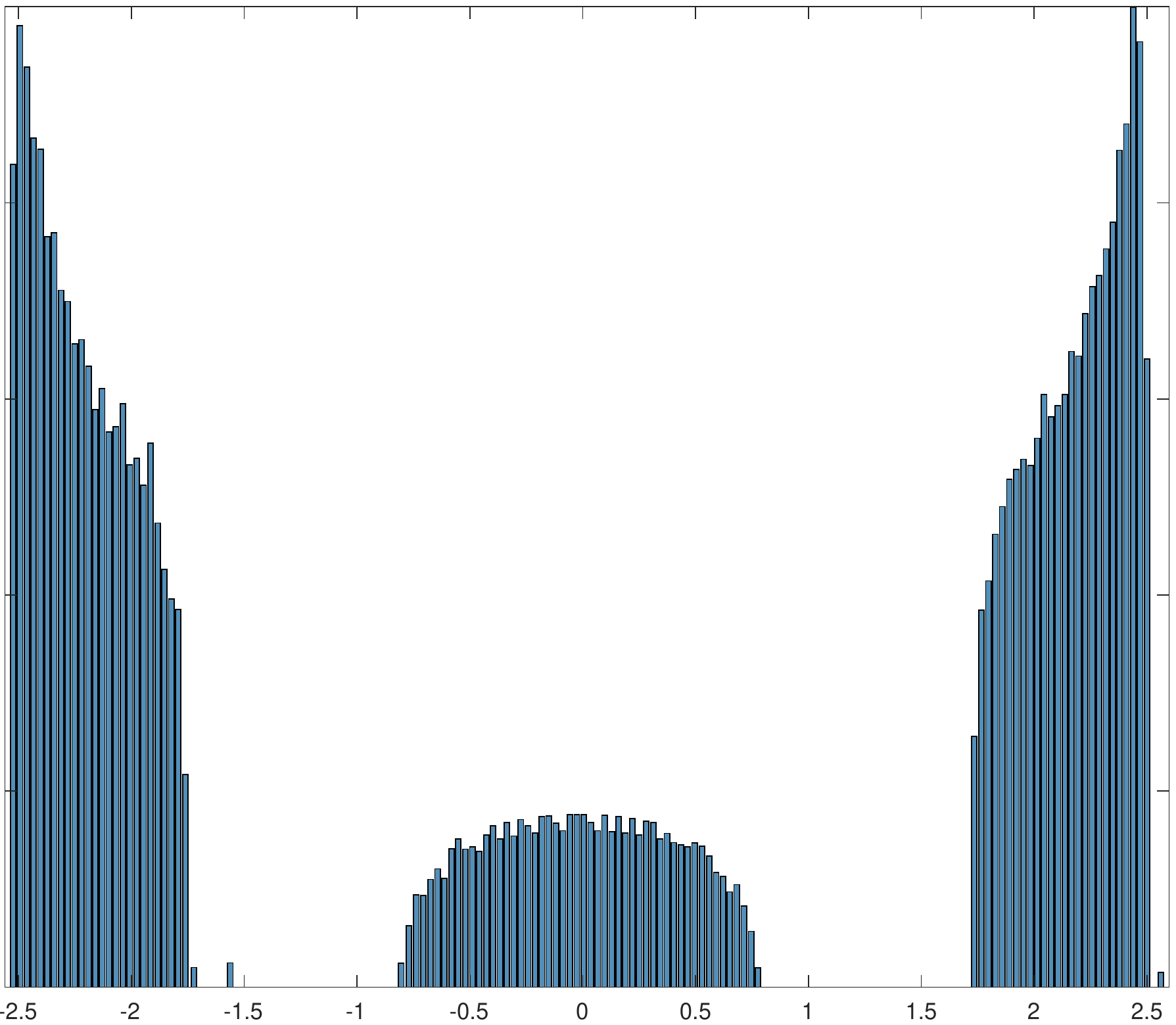}
         \\ $G$ & $\mu_{u_1}=\mu_{u_2}=\mu_{u_4}=\mu_{u_5}$ \hspace{.6cm} & $\mu_{u_3}$
         \end{tabular}
         
 \caption{On the left a finite graph $G$, where the vertices in $X_1(G)=X_{-1}(G)$ are colored in red.  In the middle and on the right we show the two distinct spectral measures of $A_\T$ associated to the vertices of $G$.  }
    \label{fig:spectralmeasures}
\end{figure}

\begin{example}
In Figure \ref{fig:spectralmeasures}, $G$ is a finite graph with $a_e\equiv 1$ and $b_v \equiv 0$ for every $e\in E(G)$ and $v\in V(G)$. By the symmetries in $G$, the spectral measures corresponding to the vertices $u_1, u_2, u_4$ and $u_5$ are equal. Hence, $A_\T$ has only two distinct spectral measures associated to the vertices of $G$, sketched in Figure \ref{fig:spectralmeasures}. These sketches were generated by taking a random lift of $G$ of degree 1200 and by plotting the weighted histogram for the corresponding discrete spectral measures. As the figures show, the spectral measures corresponding to  $u_1, u_2, u_4$ and $u_5$ have  atoms  at $-1$ and 1, while the spectral measure corresponding to $u_3$ is absolutely continuous with respect to the Lebesgue measure on $\mathbb{R}$. Hence, $X_{-1}(G)=X_1(G)= \{u_1, u_2, u_4, u_5\}$. Note that the subgraph induced by $X_1(G)$ consists of two disconnected trees; later in Theorem \ref{thm:aomoto} we will show that this a general property of Aomoto sets. 
\end{example}

We will use repeatedly an equivalent form of Proposition \ref{prop:equivallencesforeigs} \eqref{prop:equivallencesforeigs-v} above: \emph{if $u \notin \aom{\lambda}{G}$, then any eigenvector $\eta \in \Ker(\lambda - \jacobi{\T})$ is identically zero on the fibre over $u$}. We will also require a standard identity expressing the mass assigned to $\lambda \in \Spec_p \jacobi{\T}$ by the spectral measure $\mu_u$.

\begin{lemma}
    \label{lem:atom-basis}
    Let $G$ be a finite graph, $\T$ be its universal cover, and $\lambda \in \Spec_p \jacobi{\T}$. Then if $\frB$ is any orthonormal basis for $\Ker(\lambda - \jacobi{\T})$, for any $u\in V$ and $\tilde u \in \Xi^{-1}(u)$,
    \begin{equation}
        \label{eq:mass-basis}
        \mu_u\{\lambda\} = \sum_{\eta \in \frB} |\eta(\tilde u)|^2.
    \end{equation}
\end{lemma}

\noindent Equation \eqref{eq:mass-basis} follows from a standard application of the Borel functional calculus, where the key observation is that if $f_\lambda :\Spec A_\T\to \bbR$ is the indicator function of the singleton $\{\lambda\}$ then $f_{\lambda}(A_\T)$ is the orthogonal projection onto $\ker(\lambda -A_\T)$. 

%% file: Content/main-results.tex
\section{Main Results}
\label{sec:mainresults}

Our first contribution is to strengthen a result of Aomoto \cite{aomoto1991point}, by way of a new and more conceptual proof. This result characterizes the induced subgraph on $\aomoto{\lambda}{G}$ for any $\lambda \in \Spec_p \jacobi{\calT}$, and relate the mass $\mu\{\lambda\}$ to the local structure of this subgraph and neighboring vertices. Let us write $\partial \aomoto{\lambda}{G}$ for the set of vertices outside the Aomoto set but connected to it by an edge, $\cc \aom{\lambda}{G}$ for the number of connected components of the subgraph induced by $\aom{\lambda}{G}$, and define the \textit{index} of $\lambda$ as
\begin{equation}
    \label{eq:index}
    \ind{\lambda}{G} = \cc \aom{\lambda}{G} - |\partial \aom{\lambda}{G}|.
\end{equation}
\noindent Recall that for us a graph $G = (V,E,a,b)$ contains vertices $V$, directed edges $E$, Hermitian edge weights $a : E \to \bbC$ satisfying $\conj{a_e} = a_{\check e}$ and real potential $b : V \to \bbR$.

\begin{theorem}
    \label{thm:aomoto}
    Let $G$ be a finite graph, $\T$ be its universal cover, and $\lambda \in \Spec_p \jacobi{\T}$. Then:
    \begin{enumerate}[(i)]
        \item \label{thm:aomoto-i} The subgraph induced by $\aom{\lambda}{G}$ is acyclic,
        \item \label{thm:aomoto-ii} $\lambda$ is an eigenvalue, with multiplicity one, of the induced Jacobi operator of each connected component of this subgraph, and
        \item \label{thm:aomoto-iii} The density of states of $\jacobi{\T}$ satisfies
        \begin{equation}
            \label{eq:aom-formula}
            \mu\{\lambda\} = \frac{\ind{\lambda}{G}}{|V(G)|}
        \end{equation}
    \end{enumerate}
\end{theorem}

\noindent Assertion \eqref{thm:aomoto-i}, claimed without proof in \cite{avni2020periodic}, clarifies an ambiguity in Aomoto's result, which did not rule out self-loops or multi-edges in the subgraph induced by $\aom{\lambda}{G}$; \eqref{thm:aomoto-ii} is a new observation, and \eqref{thm:aomoto-iii} is due to Aomoto.  Our new proof is combinatorial and linear algebraic, using properties of eigenvectors of Jacobi operators on finite and infinite trees; the question of finding an alternative to Aomoto's original proof explaining the significance of the quantity $\ind{\lambda}{G}$, was posed in \cite[Problem 8.1]{avni2020periodic}. The proofs of \eqref{thm:aomoto-i} and \eqref{thm:aomoto-ii} may be found in Section \ref{sec:acyclic}, and that of \eqref{thm:aomoto-iii} in Section \ref{sec:mass}.

We then build on the proof of Theorem \ref{thm:aomoto} to prove a number of novel results. First, we show that for any graph $G$, the point spectrum of the periodic Jacobi operator on its unversal cover is contained in $\Spec \jacobi{G}$---with multiplicity bounded in terms of the index $\ind{\lambda}{G}$. In fact, we can further refine this result for the Jacobi operator of any cover $H$ of $G$.

\begin{theorem}
    \label{thm:multiplicity}
    Let $G$ be a finite graph, $H$ an $n$-lift of $G$, and $\T$ their common universal cover. If $\lambda \in \Spec_p \jacobi{\T}$, then
    \begin{enumerate}[(i)]
        \item \label{thm:multiplicity-i} $\lambda \in \Spec \jacobi{G}$ with multiplicity at least $|V(G)| \cdot \mu\{\lambda\}$, and
        \item $\lambda \in \Spec \jacobi{H/G}$ with multiplicity at least $(n - 1) |V(G)|\cdot \mu\{\lambda\}$,
    \end{enumerate}
    so that the multiplicity of $\lambda \in \Spec \jacobi{H}$ is at least $n|V(G)|\cdot\mu\{\lambda\}$.
\end{theorem}

\noindent We will show at the end of Section \ref{example:biregular} that these lower bounds on multiplicity need not be tight.

Additionally, we prove a converse to Theorem \ref{thm:aomoto}, namely that if a graph has a set replicating the structure of the Aomoto set for some $\lambda$, then its universal cover has $\lambda$ in its point spectrum. To be precise, let us  extend the notation $\partial$ and $\cc$ to apply to any  $X\subset V(G)$, and for each $\lambda \in \bbR$, let $\calA_\lambda(G)$  be the set of all subsets $X \subset V(G)$ which induce an acyclic subgraph, each connected component of which has $\lambda$ in the spectrum of its induced Jacobi operator and such that $\cc(X)-|\partial X|>0$. 

\begin{theorem}
    \label{thm:aomoto-converse}
    Let $G$ be a finite graph, and let $\T$ be its universal cover. For any $\lambda \in \bbR$ and $X \in \calA_\lambda(G)$, 
    \begin{enumerate}[(i)]
        \item \label{thm:aomoto-converse-i}$\lambda \in \Spec \jacobi{G}$, with multiplicity at least $\cc(X) - |\partial X|$.
        \item \label{thm:aomoto-converse-ii} \label{thm:aomoto-converseii} $\lambda \in \Spec_p \jacobi{\T}$, and $|V(G)|\cdot \mu \{\lambda\} \ge \cc(X) - |\partial X|$.
    \end{enumerate}
\end{theorem}

\noindent The lower bounds in \eqref{thm:aomoto-converse-i} and \eqref{thm:aomoto-converse-ii} need not be tight. For \eqref{thm:aomoto-converse-i}, this is shown at the end of Section \ref{example:biregular}; for \eqref{thm:aomoto-converse-ii}, when $I_\lambda (G)>1$, one may choose $X$ to contain only a subset of the trees in the true Aomoto set for some $\lambda \in \Spec_p \jacobi{\T}$. Furtheremore, there are cases where the inequality in (\ref{thm:aomoto-converse-ii}) is strict for elements in $\mathcal{A}_\lambda(G)$ that are not subsets of the Aomoto set. 

The proofs of both Theorem \ref{thm:multiplicity} and Theorem \ref{thm:aomoto-converse} follow from a generalization of the latter, Theorem \ref{thm:aomoto-converse-unitary}, which we state and verify in Section \ref{sec:gen-converse}. The argument proceeds, roughly, by patching together and extending the $\lambda$-eigenvectors on each component of the Aomoto set promised by Theorem \ref{thm:aomoto} to global eigenvectors of $\jacobi{G}$ and $\jacobi{H/G}$. This has the interesting consequence that if $\lambda \in \Spec_p A_\T$ and $G_1, G_2, \dots$ is a sequence of lifts of $G$, there is a constant fraction of $|V(G_n)|$ of  $\lambda$-eigenvectors of $A_{G_n}$ that are localized. In contrast, under some technical assumptions, quantum ergodicity results \cite{anantharaman2019quantum} imply that if $\lambda$ is in the absolutely continuous part of the spectrum of $A_\T$, the proportion between the number of localized $\lambda$-eigenvectors  of $A_{G_n}$ and $|V(G_n)|$  goes to zero. 

Combining Theorems \ref{thm:multiplicity} and \ref{thm:aomoto-converse}, we find the following corollary. Note that since there are finitely many induced subgraphs, in finite time we can find every $\lambda \in \bbR$ for which $\calA_\lambda(G)$ is nonempty.

\begin{corollary} 
\label{cor:finitealg}
    Let $G$ be a finite graph, $\T$ it's universal cover. Then for each $\lambda \in \bbR$,
    \begin{equation} \label{eq:massoflambdaformula}
        \mu\{\lambda\} = \frac{1}{|V(G)|}\max_{X \in \calA_\lambda(G)}\big(\cc(X) - |\partial(X)|\big).
    \end{equation}
Moreover, $\Spec_p \jacobi{\T}$ may be computed from $G$ in finite time.
\end{corollary}

\noindent Although Theorem \ref{thm:aomoto} implies that the Aomoto set $X_\lambda(G)$ is an element in $A_\lambda(G)$ maximizing the quantity on the right side of \eqref{eq:massoflambdaformula}, it can happen that there is not a unique maximizer. Figure \ref{fig:twomaxsets} gives such an example.

\begin{figure}[h]
    \centering
   \includegraphics[scale=.65]{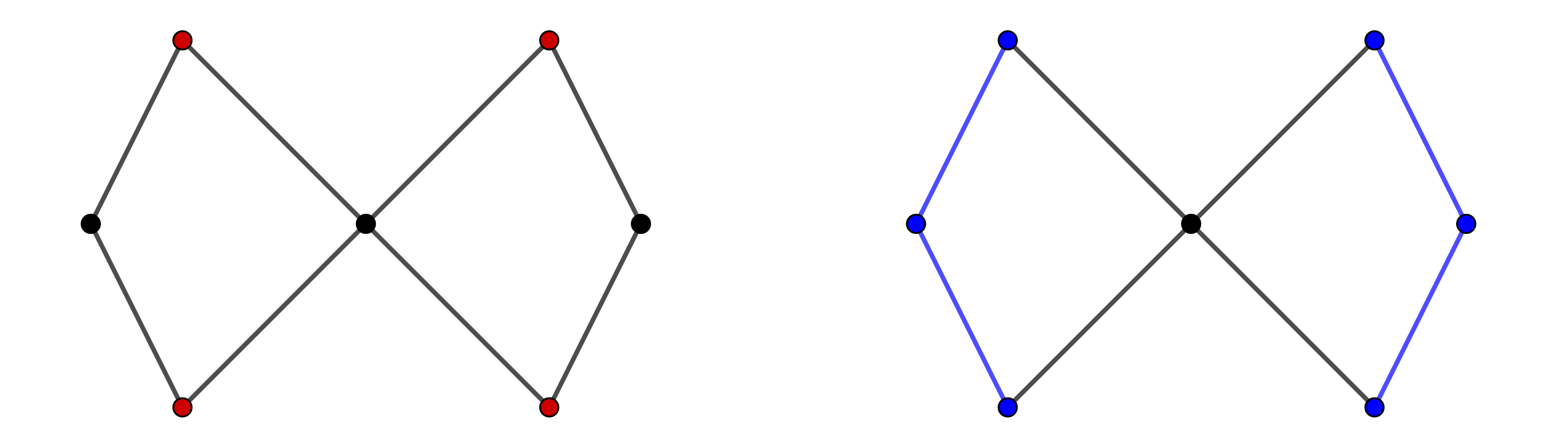}
    \caption{Two distinct sets of vertices (in red and blue respectively) of a graph $G$ are shown. If $a\equiv 1$ and $b\equiv 0$, it is easy to show from Corollary \ref{cor:finitealg} that $I_0(G)=1$. Then both the red and the blue vertex set belong to $\mathcal{A}_0(G)$. It will follow from Observation \ref{lem:aomoto-trees-zero} below that $X_0(G)$ is precisely the set indicated by the red vertices.  }
    \label{fig:twomaxsets}
\end{figure}

Finally, we use Theorem \ref{thm:aomoto} and Theorem \ref{thm:aomoto-converse} to argue that point spectrum is rare in a certain sense. To make this precise fix $G=(V, E)$ and think of the set of possible Hermitian edge weights $a=(a_e)_{e\in E}$ and vertex potentials $b=(b_v)_{v\in V}$ as $\C^{|E|/2}\oplus \R^{|V|} \cong \R^{|E|+|V|}$.   

\begin{theorem}
\label{thm:setofparameterswithpp}
    Let $G=(V, E)$ be a finite graph with at least one cycle and $\T$ be its universal cover. Assume that every vertex in $G$ has at least $d_{\min}$ distinct neighborhs. Leaving $V$ and $E$ fixed, let $\P \subset \R^{|E|+|V|}$ be the set of Hermitian edge weights and potentials for which $\Spec_p A_\T \neq \emptyset$. Then, $\P$ is a semialgebraic closed set of codimension at least $\max \{d_{\min}-1, 1\}$.\footnote{In a previous version of this paper we only proved that $\P$ is a closed set of Lebesgue measure 0 by  showing that $\P$ was contained in an algebraic set of codimension 1. We thank Barry Simon for pointing out this stronger version of the theorem and suggesting a sketch of the proof.} 
\end{theorem}

\begin{remark}
Even if the bound $\mathrm{codim} (\P) \geq \max \{d_{\min}-1, 1\}$ is tight in general, for many specific instances a stronger bound can be obtained. We refer the reader to the discussion in Section \ref{sec:spectral-deloc} for a stronger bound that depends in a more complicated way on the  combinatorial structure of $G$.  
\end{remark}

\noindent Theorem \ref{thm:setofparameterswithpp} will be proved in Section \ref{sec:spectral-deloc} and resolves \cite[Question 2]{aomoto1991point}, which speculated that the existence of point spectrum was dependent on the combinatorial structure of $G$ and not on the edge weights and potential. Results in a similar direction were obtained in \cite{keller2012absolutely} and \cite{keller2013spectral}. Their results are less general in the sense that they require $G$ to have edge weights $a \equiv 1$ and a loop at every vertex. However, they allow for more general potentials on the more general class of trees with finite cone type.

Theorem \ref{thm:setofparameterswithpp} implies in particular that $\P^c$ is an open dense set and hence that the point spectrum of $A_\T$ can be destroyed by adding arbitrarily small perturbations, even when $A_\T$ has isolated eigenvalues. This is surprising, given a result of Kato (see \cite[Section XII.2]{simon1978methods}) that if  $H$ is a bounded self-adjoint operator and $\lambda \in \Spec_p H$ is isolated with $\dim \ker(\lambda-H) <\infty$, then every sufficiently small perturbation of $H$ has non-empty point spectrum. Of course, this does not contradict our result, since Proposition \ref{prop:equivallencesforeigs} ensures that every $\lambda \in \Spec_p A_\T$ has an infinite-dimensional eigenspace. However, it is not the case that infinite-dimensional eigenspaces are unstable in general, and in many cases the phenomenon implied by Kato's result is still present.

Furthermore, it is an immediate consequence of Theorem \ref{thm:setofparameterswithpp} that $\P$ has Lebesgue measure zero. This can interpreted as an almost sure spectral delocalization result, since it implies that under a random absolutely continuous perturbation (with respect to the Lebesgue measure) of the edge weights and potential of $G$, the spectrum of $A_\T$ becomes purely absolutely continuous.

We conclude this section by giving some applications of the results presented above. 

\subsection{Point Spectrum of Biregular Trees}
\label{example:biregular}
    Let $G$ be the complete bipartite graph $K_{c, d}$ for some integers $d> c$, and denote by $V_c$ and $V_d$ the vertex components of $G$ having $c$ and $d$ vertices respectively. We will first analyze the case when $b\equiv 0$ and $a$ is any Hermitian edge weighting. It is easy to see that $V_d$ is a set satisfying the conditions of Theorem \ref{thm:aomoto-converse} for $\lambda =0$, and hence that $\mu\{0\}\geq \frac{d-c}{d+c}$. Then by Theorem \ref{thm:aomoto}, $X_0(G)\neq \emptyset$ and moreover $I_0(G)\geq d-c$. In the forthcoming Observation \ref{lem:aomoto-trees-zero}, we will show that that when $b\equiv 0$ (regardless of the structure of $G$) the set $X_0(G)$ is an independent set in $G$,\footnote{By an independent in $G$ we mean a set  $X\subset V(G)$ which induces a subgraph with no edges.} which together with the previous observations implies that in fact $X_0(G) = V_d$.  So by Theorem \ref{thm:aomoto}, when $b\equiv 0$, $\mu\{0\}=\frac{d-c}{d+c}$ and the vectors in $\ker(A_\T)$ are supported on the fibre of $V_d$, which extends a result of Godsil and Mohar \cite{godsil1988walk}. 
    
    In the case of arbitrary real potential $b$, the existence and location of eigenvalues of $\jacobi{\T}$ depend on the particular choice of $b$, and moreover by Theorem \ref{thm:setofparameterswithpp} we know that one may choose $b$ such that $A_\T$ has no point spectrum. This discussion resolves Problems 8.6 and 8.7 posed in \cite{avni2020periodic}. Finally, we note in passing that in this case, when $b\equiv 0$ we have $I_0(G)=d-c$ but the multiplicity of zero in $\Spec A_G$ is $d+c-2$, which shows that the bounds on the multiplicity given in Theorems \ref{thm:multiplicity} and \ref{thm:aomoto-converse} may not be tight.  

\subsection{Non-isolated Point Spectrum}
    Let $G$ be a finite graph and let $n=|V(G)|$.  Sunada's gap labeling theorem (see\cite[Theorem 5.1]{avni2020periodic} or \cite[Theorem 1.8]{vargas2019spectra}) states that $\Spec \jacobi{\T}$ is a disjoint union of at most $n$ (possibly degenerate) closed intervals typically called \textit{bands}, and that if $B$ is one of these bands then $\mu(B) = j/n$ for some $j\in [n]$. If $\lambda\in \Spec_p A_\T$ is isolated then $\{\lambda\}$ is a (degenerate) band of $\Spec A_\T$. This is the case, for example, when $\lambda =0$, $G$ is a bipartite biregular graph with components of different sizes, $a\equiv 1$ and $b\equiv 0$. Here we will show that it is possible for $0\in \Spec_p A_\T$ to lie inside a non-degenerate band of $\Spec A_\T$. Our argument is similar in spirit to the one used at the end of Section 5 of \cite{avni2020periodic} for an unrelated purpose. 
    
    Set $b \equiv 0$, noting that $\mu$ is symmetric about zero from Lemma \ref{lem:symmetryofspectrum}, and assume that $0\in \Spec_p A_\T $ and that $n-I_0(G)$ is an odd number. Figure \ref{fig:nonisolatedpps}  shows two instances where these conditions are met. If $\{0\}$ is an isolated point of $\Spec A_\T$ then any band in $\Spec A_\T$ is either disjoint from $(0, \infty)$ or fully contained in this infinite interval. Hence, by Sunada's gap labeling theorem $\mu(0, \infty) = j/n$ for some integer $j\in [n]$. On the other hand, since $\mu$ is symmetric, $\mu(-\infty, 0) = j/n$. Finally, by Theorem \ref{thm:aomoto} $\mu\{0\} = I_0(G)/n$. Putting all these observations together we get $1=\mu(\R) = \frac{2j+I_0(G)}{n},$ which contradicts the assumption that $n-I_0(G)$ is odd. 

\begin{figure}[h]
    \centering
   \includegraphics[scale=.65]{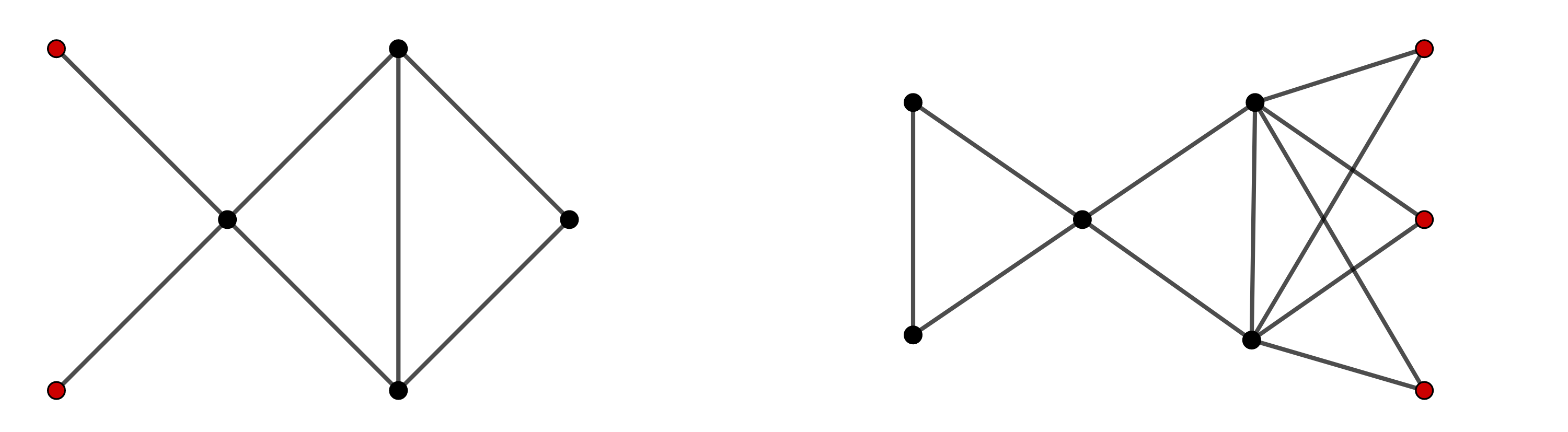}
    \caption{As in Section \ref{example:biregular}, a combination of Observation \ref{lem:aomoto-trees-zero} and Theorem \ref{thm:aomoto-converse} yields that the red vertices are the Aomoto set associated to 0 for each of the graphs displayed above. In both cases $I_0(G)=1$ while $|V(G)|$ is even.}
    \label{fig:nonisolatedpps}
\end{figure}

%% file: Content/acyclictree.tex
\section{Acyclic Nature of Aomoto Sets}
\label{sec:acyclic}

In this section we will prove the first two assertions of Theorem \ref{thm:aomoto}, namely that if $\lambda \in \Spec_p \T$, then the Aomoto set $\aom{\lambda}{G}$ is acyclic, and $\lambda$ is an eigenvalue of the induced Jacobi operator on each of its connected components. We begin by generalizing to the infinite case a result of Fielder regarding eigenvectors of finite trees \cite[Proposition 1]{fiedler1975eigenvectors}.

\begin{lemma}
    \label{lem:treemultiplicity}
    Let $T$ be a locally finite tree with Hermitian edge weights and potential $a : E(T) \to \bbC$ and $b : V(T) \to \bbR$ respectively, and Jacobi operator $\jacobi{T}$. If $\eta \in \Ker(\lambda - \jacobi{t})$ and $\eta(v) \neq 0$ for every vertex $v \in V(T)$, then $\dim \Ker(\lambda - \jacobi{T}) = 1$.
\end{lemma}

\begin{proof}
    Choose a root $r$ for $T$, and for each vertex $v$, write $p(v)$ for its unique parent, $T_v$ the infinite sub-tree emanating from $v$ away from its parent and $\eta|_{\ge v}$ for the restriction of $\eta$ to the subtree $T_v$. As $T$ is acyclic, it has no multi-edges or self-loops, and there is no ambiguity in writing $a_{v \from u}$ for the weight of the unique edge with source $u$ and terminal $v$. We then have
    \begin{align*}
        (\lambda - \jacobi{T})\eta|_{\ge v} 
        &= \sum_{u \in V(T_v)} \lambda \eta(u) \delta_u - \sum_{u \in V(T_v) \setminus \{v\}} (\jacobi{T} \eta)(u)\delta_u \\
        &\qquad \qquad \qquad \,\,\,\, - \left(b_v \eta(v) + \sum_{x : p(x) = v} a_{v \from x}\eta(x) \right)\delta_v - a_{p(v) \from v}\eta(v)\delta_{p(v)} \\
        &= \lambda \eta(v)\delta_v - \left((\jacobi{T}\eta)(v) - a_{v \from p(v)}\eta(p(v)) \right) \delta_v-
        a_{p(v) \from v)}\eta(v)\delta_{p(v)} \\
        &= a_{v \from p(v)}\eta(p(v))\delta_v - a_{p(v) \from v}\eta(v)\delta_{p(v)}
    \end{align*}
    for any $v \neq r$. Now, let $\zeta \in \Ker (\lambda - A_T)$. As $\jacobi{T}$ is self-adjoint and $\lambda$ real,
    \begin{align*}
        0 = \langle \zeta, (\lambda - \jacobi{T})\eta|_{\ge v}\rangle 
        = a_{v \from p(v)}\eta(p(v))\conj{\zeta(v)} - a_{p(v) \from v}\eta(v) \conj{\zeta(p(v))},
    \end{align*}
    which implies
    $$
         \frac{\zeta(v)}{\zeta(p(v))} =  \conj{\frac{a_{p(v) \from  v}}{a_{v \from p(v)}} \frac{\eta(v)}{\eta(p(v))}}.
    $$
    This identity holds for every $\zeta \in \Ker(\lambda - A_\calT)$, including $\eta$ itself, so we obtain
    $$
        \frac{\zeta(v)}{\zeta(p(v))} =  \conj{\frac{a_{p(v) \from  v}}{a_{v \from p(v)}} \frac{\eta(v)}{\eta(p(v))}} = \conj{\frac{a_{p(v) \from v}}{a_{v \from p(v)}} \conj{\frac{a_{p(v) \from v}}{a_{v \from p(v)}} \frac{\eta(v)}{\eta(p(v))}}} = \frac{|a_{p(v) \from v}|^2}{|a_{v \from p(v)}|^2} \frac{\eta(v)}{\eta(p(v))} = \frac{\eta(v)}{\eta(p(v))};
    $$
    in the final equality we have used conjugate symmetry of the edge weights. Since $\eta|_{\ge r} = \eta \in \Ker(\lambda - A_\calT)$, $\zeta$ is unconstrained at the root, and the above equation propagates a condition down the tree that $\zeta = \eta \cdot \zeta(r)/\eta(r)$.
\end{proof}

We now prove that the subgraph of $G$ induced by $\aom{\lambda}{G}$ is a forest, and that $\lambda$ is an eigenvalue, with multiplicity one, of the induced Jacobi operator of each of its connected components.

\begin{proof}[Proof of Theorem \ref{thm:aomoto}(\ref{thm:aomoto-i}-\ref{thm:aomoto-converse-ii})]
    Assume $\lambda$ is in the point spectrum of $\jacobi{\T}$, and let $G'$ be a connected component of the subgraph induced by $\aom{\lambda}{G}$. Let $\T'$ be the universal cover of $G'$. If we view $\T'$ as a subgraph of $\T$ then any vector in $\Ker(\lambda - \jacobi{\T})$ vanishes on the boundary of $\T'$ in $\T$, and thus restricts to a $\lambda$-eigenvector of $\T'$. Hence $\aom{\lambda}{G'} = V(G')$ by Proposition \ref{prop:equivallencesforeigs}\eqref{prop:equivallencesforeigs-v}, and we can now use the following observation, which follows from Zorn's lemma and appeared as \cite[Lemma 7]{nylen1998null}.
    
    \begin{observation}
        \label{obs:nonzeroentries}
        If $\aom{\lambda}{G'} = V(G')$ then there is an $\eta\in \Ker(\lambda - \jacobi{\T'})$ satisfying $\eta(u)\neq 0$ for every $u\in V(\T')$. 
    \end{observation}
    
    \noindent Combining Observation \ref{obs:nonzeroentries} and Lemma \ref{lem:treemultiplicity} we conclude finally that $\dim \Ker(\lambda - \jacobi{\T'}) = 1$, and thus, by Proposition \ref{prop:equivallencesforeigs}\eqref{prop:equivallencesforeigs-i}, that $G'$ is acyclic. This further implies that $\T' = G'$, which proves the second assertion.
\end{proof}

In the course of the proof above we showed the following fact, which will be of repeated use throughout the paper. 

\begin{lemma}
    \label{lem:nonzeroAomoto}
    Let $G$ be a finite graph with Hermitian edge weights and potential, with $\lambda \in \Spec_p A_\T$ and $T_1, \dots, T_p$    the Aomoto trees of $G$ associated to $\lambda$. Then, for every $i \in [p]$ there is a unique (up to phase) unit vector $\zeta_i \in \Ker(\lambda - \jacobi{T_i})$ satisfying $\zeta(u)\neq 0$ for every $u\in V(T_i)$.
\end{lemma}

For use in the next section, we record one  consequence of the above lemma.

\begin{observation}
    \label{lem:aomoto-trees-zero}
    Let $G$ be a graph with $b\equiv 0$, $\T$ its universal cover, and assume $0 \in \Spec_p \jacobi{\T}$. Then $\aom{0}{G}$ is an independent set in $G$.
\end{observation}

\begin{proof}
    By Lemma \ref{lem:nonzeroAomoto},  each Aomoto tree of $G$ must have a unique, everywhere nonzero eigenvector in the kernel of its Jacobi operator. On the other hand, a vector in the kernel of a Jacobi operator with potential zero for a tree cannot be nonzero at the parent of a leaf. Thus every Aomoto tree of $G$ is an isolated vertex as desired.
\end{proof}

%% file: Content/mass_formula.tex
\section{Aomoto's Index Formula}
\label{sec:mass}

In this section we complete the proof of Theorem \ref{thm:aomoto} by verifying the formula in equation \eqref{eq:aom-formula}: if $\lambda \in \Spec_p \jacobi{\T}$, then
$$
    |V(G)| \cdot \mu\{\lambda\} = \ind{\lambda}{G}.
$$
Our strategy will be to reduce the problem to the proof of an analogous result on an auxiliary bipartite graph $G'$.

\subsection{Constructing the Auxilliary Graph} 

Let $T_1, \dots, T_p $ be the Aomoto trees of $G=(V, E, a, b)$ associated to $\lambda$, write $\mathcal{F}_i$ for the set of disjoint copies of $T_i$ in $\T = (\calV,\calE,\a,\b)$ obtained by lifting $T_i$, and let $\F = \bigcup_{i=1}^p \F_i$. Note that $\F$ is a subforest of $\T$ and all of its subtrees are isomorphic to some Aomoto tree of $G$. 

By Lemma \ref{lem:nonzeroAomoto} there is for each $T_i$ a unique (up to phase) vector $\zeta_i \in \Ker(\lambda - \jacobi{T_i})$ with unit norm and nonzero entries. Take any $\eta \in \Ker(\lambda - \jacobi{\T})$. For every $S\in \F$, by definition of the Aomoto set it holds that  $\eta$ is zero on all vertices in $\partial V(S)$. Hence, the restriction of $\eta$ to any $S\in \F_i$ induces an eigenvector of $A_{T_i}$. This implies that $\eta$ can be decomposed as
\begin{equation}
    \label{eq:eigdecomposition}
    \eta = \sum_{S\in \F } \alpha_S \zeta_S,
\end{equation}
where $\alpha_S\in \R$ are coefficients and the $\zeta_S \in \ell^2(V(\T))$ are inclusions of the $\lambda$-eigenvectors of each Aomoto tree:
$$
    \zeta_S(v) = 
    \begin{cases} 
        \zeta_i(\Xi(v)) & \text{if } v\in V(S) \text{ and } S\in \mathcal{F}_i 
        \\ 0 &\text{otherwise} 
    \end{cases}.
$$
We now construct $G' = (V',E',a',b')$; the process is summarized in Figure \ref{fig:auxgraph}. First, $V'$ is obtained from $V$ by deleting every vertex outside $\aom{\lambda}{G} \cup \partial \aom{\lambda}{G}$, and contracting each Aomoto tree $T_i$ to a single vertex $t_i$. Write $\{t_1,...,t_p\} = U \subset V'$, and identify $\partial \aom{\lambda}{G}$ with $\partial U$. Now, for each $v \in \partial U = \partial \aom{\lambda}{G}$ and each edge $e \in \tau(v) \subset E$ whose source is in a tree $T_i$, include an edge $e' \in E'$ with $\tau(e') = v$ and $\sigma(e') = t_i$, and set its weight as
\begin{equation}
    \label{eq:new-edge-weights}
    a_{e'} = a_e \zeta_i(\sigma(e)).
\end{equation}
This process is mirrored to construct an edge $f' \in E'$ from any $f \in \sigma(v) \subset E$ whose terminal is in $T_i$; no other edges are included. Finally, the potential $b'$ is identically zero. 

\begin{figure}
    \centering
   \includegraphics[scale=.85]{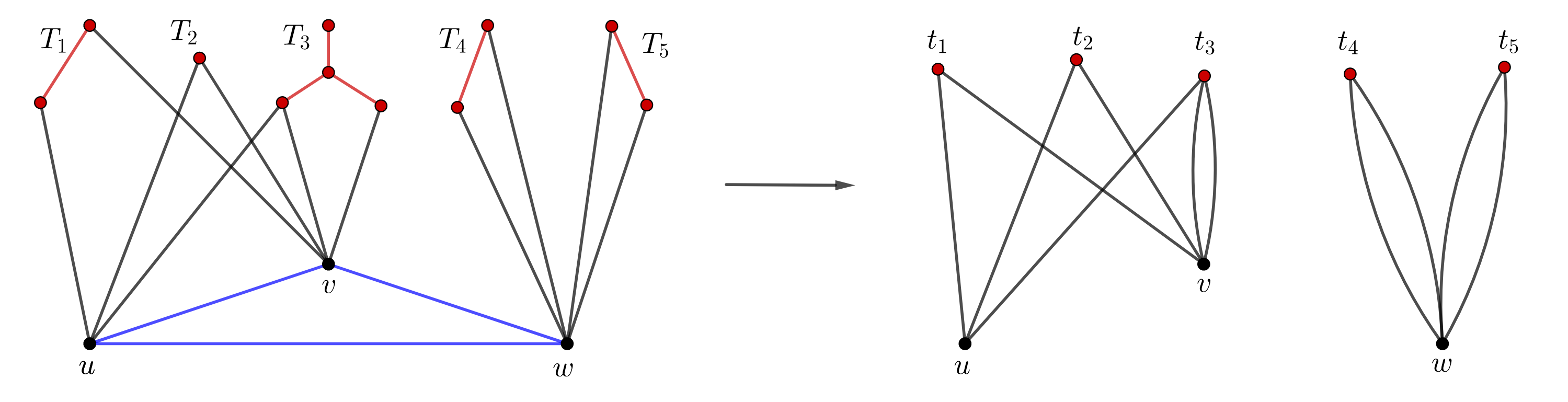}
    \caption{On the left an example of a graph $G$ with Aomoto trees in red. On the right its auxiliary graph $G'$, where each tree $T_i$ has been contracted into a vertex $t_i$ and the blue edges have been removed.  }
    \label{fig:auxgraph}
\end{figure}

We have arranged things so that $G'$ is bipartite, with connected components $G'_1,...,G'_m$, whose respective covers we will denote $\T'_1,...,\T'_m$. We may also construct a new infinite graph $\T' = (\calV',\calE',\a',\b')$ from $\T$ analogously to the construction of $G'$ from $G$: by deleting the fibres over any vertex outside $\aom{\lambda}{G}\cup \partial \aom{\lambda}{G}$, contracting each tree $S \in \F \subset \T$ into a single vertex $u_S$, including for each $e \in \calE$ with ends in $\Xi^{-1}(\aom{\lambda}{G})$ and $\Xi^{-1}(\partial \aom{\lambda}{G})$ a corresponding edge $e' \in \calE'$ with ends in the contraction if $\Xi^{-1}(\aom{\lambda}{G})$ and its boundary, and reweighting any such edge according to \eqref{eq:new-edge-weights}. This $\T'$ consists of countably many copies of each $\T'_j$, and is a cover of $G'$ via a map $\Xi'$; note that $(\Xi')^{-1}(U) = \{u_S : S \in \calF\}$, the contraction of $\Xi^{-1}(\aom{\lambda}{G})$. With this setup and the decomposition in  \eqref{eq:eigdecomposition}, any $\eta\in \Ker(\lambda - \jacobi{\T})$ gives rise to a vector  $\eta'\in \ell^2(V(\T'))$ in a natural way:
\begin{equation}
    \label{eq:ftilde}
     \eta = \sum_{S\in \F} \alpha_S \zeta_S \mapsto \eta' =\sum_{S\in \F} \alpha_S \delta_{u_S}.
\end{equation}

\begin{observation}
    \label{obs:eta'}
    The map $\eta \mapsto \eta'$ is an isometric inclusion of $\Ker(\lambda - \jacobi{\T})$ in $\Ker \jacobi{\T'}$.
\end{observation}

\begin{proof}
    Preservation of norm is immediate since $\zeta_S$ is a unit vector, and since the map is an isometry it is injective; it remains only to show that $\Ker(\lambda - \jacobi{\T})$ is mapped to $\Ker \jacobi{\T'}$. The vector $\eta'$ is identically zero on the fibre over $\partial U$ and thus $(\jacobi{T'}\eta')(u) = 0$ for any $u \in (\Xi')^{-1}(U)$. It remains only to consider $v \in (\Xi')^{-1}(\partial U)$, which as above we may identify with $\Xi^{-1}(\partial \aom{\lambda}{G}) \subset \calV$. For each edge $e \in \tau(v) \subset \calE$ write $S_e$ for the tree in $\F$ to which $\sigma(e)$ belongs, so that the reweighting in \eqref{eq:new-edge-weights} gives $a'_{e'} = a_e \zeta_{S_e}(\sigma(e))$. As the potential $b'$ is identically zero and $\eta$ and $\eta'$ vanish outside the fibres over $\aom{\lambda}{G}$ and $U$ respectively, we have
    \begin{align*}
        (\jacobi{\T'}\eta')(v')
        &= \b_{v'} \eta'(v') + \sum_{e' \in \tau(v') \subset \calE'} \a_{e'}' \eta'(\sigma(e')) \\
        &= \sum_{\substack{e \in \tau(v) \subset \calE : \\ \sigma(e) \in \Xi^{-1}(\aom{\lambda}{G})}} \a_e \zeta_{S_e}(\sigma(e)) \alpha_{S_e} \\ 
        &= \b_v \eta(v) + \sum_{e \in \tau(v) \subset \calE} \a_e \eta(\sigma(e)) \\
        &= (\jacobi{\T}\eta)(v) \\
        &= \lambda \eta(v) = 0.
    \end{align*}
    In the third line, note that some edges in $\tau(v) \subset \calE$ have a source outside of the fibre over $\aom{\lambda}{G}$, but that $\eta$ is identically zero there.
\end{proof}

Immediately from this observation, we can conclude that $0 \in \Spec_p \jacobi{\T'}$. Moreover, as $\T'$ is comprised of disjoint copies of the $\T'_j$'s, $\jacobi{\T'}$ restricts to $\jacobi{\T'_j}$ on each one, and thus $0 \in \Spec_p \jacobi{\T'_j}$ for at least one $\T'_j$. Our next observation characterizes the associated Aomoto set on $G'_j$. Recall that $G'$ is bipartite with vertex classes $U$ and $\partial U$, and let us write $U_j$ and $\partial U_j$ for the corresponding classes of vertices in each connected component $G'_j$.

\begin{observation}
    \label{obs:aomoto-of-aux-graph}
    $X_0(G_j') = U_j$. 
\end{observation}
 
\begin{proof}
    By Proposition \ref{prop:equivallencesforeigs}, the definition of the map $\eta \mapsto \eta'$, and Observation \ref{obs:eta'}, we immediately have the inclusion $U_j \subset X_0(G_j')$, since any $\eta \in \Ker(\lambda - \jacobi{\T})$ maps to $\eta'$ supported only on the fibre over $U_j$. On the other hand, from Lemma \ref{lem:aomoto-trees-zero}, we know that $X_0(G'_j)$ is an independent set, and $U_j$ is a maximal independent set in $G'_j$ by definition of $\partial U_j$.
\end{proof}

Finally, we can strengthen Observation \ref{obs:eta'}

\begin{observation}
    \label{obs:eta'-iso}
    The map $\eta \mapsto \eta'$ gives an isomorphism between $\Ker (\lambda - \jacobi{\T})$ and $\Ker \jacobi{\T'}$.
\end{observation}

\begin{proof}
    We noted above that $\jacobi{\T'}$ decomposes as a direct sum of the Jacobi operators on the copies of $\T'_j$ comprising $\T'$. By applying Observation \ref{obs:aomoto-of-aux-graph} separately to each copy, any $\theta \in \Ker \jacobi{\T'}$ is supported only on $(\Xi')^{-1}(U)$. Thus the adjoint of the map $\eta \mapsto \eta'$ takes any vector $\theta \in \Ker \jacobi{\T'}$ to one in $\ell^2(V(\T))$:
    $$
        \theta = \sum_{S \in \calF} \alpha_S \delta_{u_S} \mapsto  \sum_{S \in \calF} \alpha_S \zeta_S.
    $$
    Once again this is clearly injective and norm-preserving, and a parallel argument to Observation~\ref{obs:eta'} shows that it takes $\Ker \jacobi{\T'}$ into $\Ker(\lambda - \jacobi{\T})$.
\end{proof}

We can finally relate the density of states of $\jacobi{\T}$ to those of the $\jacobi{\T'_j}$.

\begin{observation}
    \label{obs:dos-aux}
    Let $T$ be an Aomoto tree in $G$, and $t \in V(G'_j)$ its contraction in a component $G'_j$ of $G'$. Writing $\mu_t'$ for the spectral measure of $t$ in $\jacobi{\T'_j}$, and for $\mu_v$ for the spectral measure of $\jacobi{\T}$ for each $v \in V(T) \subset V(G)$, we have
    $$
        \sum_{v \in V(T)} \mu_v \{\lambda\} = \mu_t'\{0\}. 
    $$
\end{observation}

\begin{proof}
    Choose a copy $\tilde T$ of $T$ in its fibre in $\T$, and let $\tilde t$ be the contraction of $\tilde T$ in $\T'_j \subset \T'$. By construction, for each $\eta \in \ker(\lambda - \jacobi{\T})$, 
    $$
        \eta'(\tilde t)^2 = \sum_{\tilde v \in V(\tilde T)} \eta(\tilde v)^2.
    $$
    Now, let $\frB'_j$ be an orthonormal basis of $\Ker \jacobi{\T'_j}$. By Observation \ref{obs:eta'-iso} this is the image of some orthonormal set $\frB_j$ in $\Ker(\lambda - \jacobi{\T})$. In particular, recalling our construction of $\T'$ from $\T$ by deleting vertices and contracting Aomoto trees, our chosen copy of $\T'_j$ in $\T'$ pulls back to a subtree $\T_j$ of $\T$ containing $\tilde T$. Moreover, $\frB_j$ is an orthonormal basis for the orthogonal projection of $\Ker(\lambda - \jacobi{\T})$ to the subspace of $\ell^2(\calV)$ supported on the vertices of $\T_j$, and we can therefore augment $\frB_j$ to an orthonormal basis $\frB$ of $\Ker(\lambda - \jacobi{\T})$, whose additional vectors vanish on $\T_j$.
    
    We now use Lemma \ref{lem:atom-basis} to compute
    \begin{align*}
        \sum_{v \in V(T)}\mu_v\{\lambda\} 
        = \sum_{\tilde v\in V(\tilde T)}\sum_{\eta \in \frB} \eta(\tilde v)^2
        = \sum_{\tilde v \in V(\tilde T)}\sum_{\eta \in \frB_j}\eta(\tilde v)^2
        = \sum_{\eta \in \frB_j} \eta'(\tilde t)^2 
        = \sum_{\eta' \in \frB_j'} \eta'(\tilde t)^2
        = \mu_t'\{0\}.
    \end{align*}
\end{proof}

\subsection{Analyzing the Auxiliary Graph} 

This section is devoted to the final observation of our proof:

\begin{observation}
    \label{obs:dos-ind}
    Fix $j \in [m]$ and assume that $0 \in \Spec_p \jacobi{\calT'_j}$. Then
    $$
        \sum_{t \in U_j} \mu'_t \{0\} = I_0(G'_j).
    $$
\end{observation}

\noindent This will finish the proof, as combining Observations \ref{obs:dos-aux} and \ref{obs:dos-ind} and recalling the construction of $G'$ gives
\begin{align*}
    |V(G)| \cdot \mu\{\lambda\} &= \sum_{u \in \aom{\lambda}{G}} \mu_u\{\lambda\} = \sum_{t \in U} \mu_t'\{0\} \\
    &= \sum_{j \in [m]} \ind{0}{G'_j} = \sum_{j \in [m]} |U_j| - |\partial U_j| \\
    &= |U| - |\partial U| = \cc \aom{\lambda}{G} - |\partial \aom{\lambda}{G}| \\
    &= \ind{\lambda}{G}.
\end{align*}

\begin{proof}[Proof of Observation \ref{obs:dos-ind}]
    Let $\mu'$ be the DOS of $A_{\T_j'}$. Let $L_1, L_2, \dots$ be a sequence of finite lifts of $G_j'$ with covering maps $\xi_n: L_n\to G_j'$. By Lemma \ref{lem:lifts-weak-conv} we may choose the $L_n$ with girth going to infinity. Since $G_j'$ is bipartite with zero potential, the Jacobi matrices $A_{L_n}$ and $A_{\T_j'}$ have the following block structure
    $$
        A_{L_n} =
        \begin{pmatrix}
        0 &  Z_n^T \\
        Z_n & 0
        \end{pmatrix} 
        \quad \text{and} \quad 
        A_{\T_j'} = \begin{pmatrix}
            0 &  Z_\infty^T \\
            Z_\infty & 0
        \end{pmatrix},
    $$
    where for $n\in \mathbb{N}\cup \{\infty\}$ the domain and range of $Z_n$ correspond to the fibers of $\partial U_j$ and $U_j$ respectively. Note that $A_{L_n}^2 = Z_n^TZ_n \oplus Z_n Z_n^T$ and $A_{\T_j'}^2 = Z_\infty^T Z_\infty \oplus Z_\infty Z_\infty^T $.
    
    Let $\mu_{Z_nZ_n^T}$ and $\mu_{Z_n^TZ_n}$ be the empirical spectral distributions of $Z_nZ_n^T$ and $Z_n^TZ_n$ respectively. Fix a positive integer $k$ and note that, since $L_n$ is bipartite, the terms in $\tr(Z_n^T Z_n)^k$ are in one-to-one correspondence with the closed walks of length $2k$ in $L_n$ that start and end at the same vertex in $\xi_n^{-1}(\partial U_j)$. Moreover, by the girth assumption, for large enough $n$ it holds that the value of the diagonal entries of  $A_{L_n}^{2k}$ are constant on each fiber $\xi_n^{-1}(v)$ for every $v\in V(G_j')$ and coincide with the respective diagonal entries of $A_{\T_j'}^{2k}$. Hence, if we write $\nu_v$ for the spectral measure of $u$ for the operator $\jacobi{\T'_j}^2$, then by the method of moments $\mu_{Z_n^TZ_n}$ and $\mu_{Z_n Z_n^T}$ converge weakly to
    \begin{equation}
        \label{eq:nu-def}
        \nu_{\partial U_j} = \frac{1}{|\partial U_j|}\sum_{v \in \partial U_j}\nu_v \qquad \text{ and } \qquad \nu_{U_j} = \frac{1}{|U_j|}\sum_{v \in U_j} \nu_v.
    \end{equation}
    
    Since $\aom{0}{G'_j} = U_j$, equation \eqref{eq:nu-def} implies $\nu_{\partial U_j}\{0\} = 0$. If it were the case that $|\partial U_j| > |U_j|$, we would have by standard properties of matrices that the spectrum of $Z_n^T Z_n$ is equal to that of $Z_n Z_n^T$, plus the eigenvalue zero with multiplicity at least $|\partial U_j| - |U_j|$, and thus
    $$
        \frac{|\partial U_j| - |U_j|}{|\partial U_j|} \le \limsup_{n \to \infty} \mu_{Z_nZ_n^T}\{0\} \le \nu_{\partial U_j}\{0\} = 0,
    $$
    a contradiction. Thus $\ind{0}{G'_j} \ge 0$. Applying the same matrix property a second time, we have
    $$
        \mu_{Z_nZ_n^T} = \left(1 - \frac{\ind{0}{G'_j}}{|U_j|}\right)\mu_{Z_n^TZ_n} + \frac{\ind{0}{G'_j}}{|U_j|}\delta_0.
    $$
    By weak convergence, and as compact measures are determined by their moments,
    $$
        \nu_{U_j} = \left(1 - \frac{\ind{0}{G'_j}}{|U_j|}\right)\nu_{\partial U_j} + \frac{\ind{0}{G'_j}}{|U_j|}\delta_0,
    $$
    and thus
    $$
        \sum_{t \in U_j} \mu'_t\{0\} = |U_j| \nu_{U_j}\{0\} = \left(|U_j| - \ind{0}{G'_j}\right)\nu_{\partial U_j}\{0\} + \ind{0}{G'_j}\delta_0\{0\} = \ind{0}{G'_j}.
    $$
\end{proof}

%% file: Content/grapheigenvalues.tex
\section{A Generalized Converse to Aomoto's Theorem}

\label{sec:gen-converse}

In this section we will prove the following generalization of Theorem \ref{thm:aomoto-converse}, and use it to prove  our other main contribution, Theorem \ref{thm:multiplicity}.

\begin{theorem}
    \label{thm:aomoto-converse-unitary}
    Let $G$ be a fintite graph, $\T$ its universal cover, $U : E(G) \to \textup{U}(n)$ a set of unitary-valued edge weights satisfying $U_e^\ast = U_{\check e}$ for every $e \in E(T)$, and $\jacobi{G,U}$ the \emph{unitary-weighted Jacobi operator} acting on $\eta \in \ell^2(V(G))\otimes \bbC^n$ as
    $$
        (\jacobi{G,U}\eta)(v) = b_v \eta(v) + \sum_{e \in \tau(v)} a_e U_e \eta(\sigma(e)) \in \bbC^n.
    $$
    If some set of vertices $X \subset V(G)$ induces an acyclic subgraph, every component of which has $\lambda$ in the spectrum of its induced unitary-weighted Jacobi operator and $\cc(X) - |\partial X| > 0$, then $\lambda \in \Spec \jacobi{G,U}$ with multiplicity at least $n(\cc(X) - |\partial X|)$.
\end{theorem}

We begin with a lemma regarding unitary-weighted Jacobi operators of finite trees.

\begin{lemma}
    \label{lem:tree-jacobi-unitary}
    Let $T$ be a finite tree, $U : E(T) \to \textup{U}(n)$ a set of unitary-valued edge weights satisfying $U^\ast_e = U_{\check e}$ for every $e \in E(T)$, and $\jacobi{T,U}$ the associated unitary-weighted Jacobi operator. If $\lambda \in \Spec \jacobi{T}$, then $\lambda \in \Spec \jacobi{T,U}$ with multiplicity at least $n$.
\end{lemma}

\begin{proof}
    As in the proof of Lemma \ref{lem:treemultiplicity}, we will choose a root $r$ of $T$, for each vertex $v$ write $p(v)$ for its unique parent and $c(v)$ for its set of children, and, since $T$ is acyclic, write $v \from u$ for the unique edge with source $u$ and terminal $v$. By absorbing $\lambda$ into the potential, it suffices to study the case when $\lambda = 0$. So, let $\eta \in \Ker \jacobi{T}$; we will produce a subspace of dimension $n$ contained in $\Ker \jacobi{T,U}$.

    Fix a vector $\zeta_0 \in \bbC^n$ and set $\zeta(r) = \zeta_0$. For each vertex $v \in V(T)$, letting $\gamma_v$ denote the directed edges in the unique shortest path from $v$ to $r$, set
    $$
        \zeta(v) = \prod_{e \in \gamma_v} U_e^\ast \cdot \eta(v) \cdot \zeta_0
    $$
    We claim that $\zeta \in \Ker \jacobi{T,U}$; since $\zeta_0$ was arbitrary, this will complete the proof. 
    
    At the root, we have 
    $$
        (\jacobi{T,U}\zeta)(r) = b_r \eta(r) \zeta_0 + \sum_{u \in c(r)} a_{r \from u} U_{r \from u} U_{u \from r} \eta(u) \zeta_0  = \left(b_r \eta(r) + \sum_{ u \in c(r)} a_{r \from u} \eta(u)\right) \zeta_0 = 0,
    $$
    since $U_{r \from u} U_{u \from r} = 1$ and  $\eta \in \Ker \jacobi{T}$. Similarly, for any other vertex $v \in V(T)$, conjugate symmetry of the unitary weights gives us
    \begin{align*}
        (\jacobi{T,U}\zeta)(v) 
        &= b_v \prod_{e \in \gamma_v} U_e^\ast \eta(v)\zeta_0 + a_{v \from p(v)} U_{v \from p(v)} \prod_{e \in \gamma p(v)} U_e^\ast \eta(p(v)) \zeta_0 + \sum_{u \in c(v)} a_{v \from u} U_{v \from u} \prod_{e \in \gamma_u} U_e^\ast \eta(u) \zeta_0 \\
        &= \left(b_v + a_{v \from p(v)}\eta(p(v)) + \sum_{u \in c(v)}a_{v \from u}\eta(u)\right)\prod_{e \in \gamma_v}U_e^\ast \zeta_0 \\
        &= 0.
    \end{align*}
\end{proof}
We can now proceed with the proof.
\begin{proof}[Proof of Theorem \ref{thm:aomoto-converse-unitary}]
    For any Aomoto tree $T$ of $G$, the induced Jacobi operator $\jacobi{T}$ has $\lambda$ in its spectrum. By Lemma \ref{lem:tree-jacobi-unitary}, the induced unitary-weighted Jacobi operator $\jacobi{T,U}$ thus satisfies $\dim \Ker(\lambda - \jacobi{T,U}) \ge n$, and therefore the space
    $$
        \bigoplus_{T \subset \aom{\lambda}{G}} \Ker(\lambda - \jacobi{T,U}) \subset \ell^2(\aom{\lambda}{G})\otimes \bbC^n \subset \ell^2(V) \otimes \bbC^n
    $$
    has dimension $n\,\cc\aom{\lambda}{G}$. We will show that it contains a subspace of dimension $n\ind{\lambda}{G}$ which is itself contained in $\Ker(\lambda - \jacobi{G}{U})$.
    
    For each $v \in \aom{\lambda}{G}$, let $\Pi_v : \ell^2(V) \otimes \bbC^n \to \ell^2(v) \simeq \bbC^n$ be the orthogonal projection to the $\bbC^n$-valued functions in $\ell^2(V)\otimes \bbC^n$ supported on $v$. For each $u \in \partial \aom{\lambda}{G}$, there is an operator
    \begin{align*}
        \phi_u &= \sum_{\substack{e \in \tau(u) \\ \sigma(e) \in \aom{\lambda}{G}}} a_e U_e \Pi_{\sigma(e)} : \bigoplus_{T \subset \aom{\lambda}{G}} \Ker(\lambda - \jacobi{T,U}) \to \ell^2(v) \simeq \bbC^n
        \intertext{and we define}
        \phi &= \bigoplus_{u \in \partial\aom{\lambda}{G}} \phi_u :  \bigoplus_{T \subset \aom{\lambda}{G}} \Ker(\lambda - \jacobi{T,U}) \to \ell^2(\partial \aom{\lambda}{G}) \simeq \bbC^{n|\partial \aom{\lambda}{G}|}.
    \end{align*}
    Counting dimensions, $\dim \Ker \phi \ge n\ind{\lambda}{G}$, and we will show that $\Ker \phi \subset \Ker(\lambda - \jacobi{G,U})$. 
    
    Let $\zeta \in \Ker\phi$; since the latter is a subspace of $\ell^2(\aom{\lambda}{G})\otimes \bbC^n \subset \ell^2(V)\otimes \bbC^n$, we have  $\zeta(u) = 0$ for every $u \notin \aom{\lambda}{G}$. This immediately gives $\left((\lambda - \jacobi{G,U})\zeta\right)(u) = 0$ for any $u$ outside the Aomoto set and its boundary, as $\zeta$ is identically zero on $u$ and its neighbors. On the other hand, if $u$ belongs to some tree $T$ in the Aomoto set, then because $\Ker \phi \subset \bigoplus_{T \subset \aom{\lambda}{G}} \Ker(\lambda - \jacobi{T,U})$ and $\zeta$ vanishes on $\partial \aom{\lambda}{G}$, we have
    \begin{align*}
        \big((\lambda - \jacobi{G,U})\zeta\big)(u) 
        &= \lambda \zeta(u) + b_u\zeta(u) + \sum_{e \in \tau(u)} a_e U_e \zeta(\sigma(e)) \\
        &= \lambda\zeta(u) + b_u\zeta(u) + \sum_{\substack{e \in \tau(u) \\ \sigma(e) \in T}} a_e U_e \zeta(\sigma(e)) = \big((\lambda - \jacobi{T,U})\zeta\big)(u) = 0
    \end{align*}
    It remains to check that $((\lambda - \jacobi{G,U})\zeta)(u) = 0$ when $u \in \partial \aom{\lambda}{G}$, which will follow from $\zeta \in \Ker \phi$. In particular, using a final time that $\zeta$ is supported only on the Aomoto set, if $u \in \partial \aom{\lambda}{G}$ we have
    \begin{align*}
        \big((\lambda - \jacobi{G,U})\zeta\big)(u) 
        &= \lambda \zeta(u) + b_u \zeta(u) + \sum_{e \in \tau(u)} a_e U_e \zeta(\sigma(e)) \\
        &= \sum_{\substack{e \in \tau(u) \\ \sigma(e) \in \aom{\lambda}{G}}} a_e U_e \zeta(\sigma(e)) = \big(\phi\zeta\big)(u) = 0.
    \end{align*}
\end{proof}

Theorems \ref{thm:multiplicity} and \ref{thm:aomoto-converse} now follow easily.

\begin{proof}[Proof of Theorem \ref{thm:multiplicity}]
    By Theorem \ref{thm:aomoto}, the Aomoto set satisfies the hypotheses of Theorem \ref{thm:aomoto-converse-unitary}, and if $H$ is an $n$-lift of $G$, both $\jacobi{G}$ and $\jacobi{H/G}$ are unitary-weighted Jacobi operators for $G$---the former with weights taking values in $U(1)$ and the latter in $U(n-1)$ by the discussion in Section \ref{subsec:jacobiops}. Thus $\lambda \in \Spec \jacobi{G}$ with multiplicity at least
    $$
        \cc(\aom{\lambda}{G}) - |\partial \aom{\lambda}{G}| = \ind{\lambda}{G} = |V(G)|\cdot \mu\{\lambda\}
    $$
    and similarly $\lambda \in \Spec \jacobi{H/G}$ with multiplicity at least $(n-1)|V(G)| \cdot \mu\{\lambda\}$, as desired.
\end{proof}

\begin{proof}[Proof of Theorem \ref{thm:aomoto-converse}]
    Assertion \eqref{thm:aomoto-converse-i} is a special case of Theorem \ref{thm:aomoto-converse-unitary}. For \eqref{thm:aomoto-converse-ii}, let $G_n$ be the sequence of lifts of $G$ promised in Lemma \ref{lem:lifts-weak-conv}, whose empirical spectral measures $\mu_{G_n}$ converge weakly to the density of states $\mu$. Applying Theorem \ref{thm:aomoto-converse-unitary} to each $\jacobi{G_n}$, viewed again as a unitary-weighted Jacobi operator on $G$, the empirical spectral measures $\mu_{G_n}$ satisfy $\mu_{G_n}\{\lambda\} \ge \frac{\cc(X) - |\partial X|}{V(G)}$. As these converge weakly to $\mu$, we have
    $$
        \mu\{\lambda\} \ge \frac{\cc(X) - |\partial X|}{V(G)}.
    $$
\end{proof}

%% file: Content/Delocalization_new.tex
\section{Spectral Delocalization for $A_\T$}
\label{sec:spectral-deloc}

In this section we will prove Theorem \ref{thm:setofparameterswithpp}. The fact that the set $\P$ mentioned in the theorem is a closed set  will follow from Theorem \ref{thm:aomoto-converse}, while the fact that it has large codimension will follow from Theorem \ref{thm:aomoto}. 

Let $G=(V, E)$ be a fixed finite and unweighted graph, for which we will vary the parameters $a_e$ and $b_v$. In what follows we will identify $\C^{|E|/2}$ with $\R^{|E|}$ so that the parameter space for the $a_e$ and the $b_v$ is a subset of $\R^{|E|+|V|}$, and we will denote elements of $\R^{|E|+|V|}$ by $(a, b)$, where $a=(a_e)_{e\in E}$ and $b =(b_v)_{v\in V}$. To ease notation define $m=|E|+|V|$.

Let $\calA(G)$ be the family of vertex sets  $X\subset V$ that induce an acyclic subgraph of $G$ with the property that  $\cc(X)-|\partial X| > 0$. For $X\in \calA(G)$  let $\calP_X\subset \bbR^{m}$ be the set of parameters for which all the Jacobi matrices of the trees induced by $X$  have a common eigenvalue. Note that Theorems \ref{thm:aomoto} (\ref{thm:aomoto-ii}) and \ref{thm:aomoto-converse} (\ref{thm:aomoto-converse-ii}) imply that
\begin{equation}
\label{eq:pointdecomp}
\P =\bigcup_{X\in \calA (G)} \P_X. 
\end{equation}
To compute the dimension of $\P$ we will analyze each $\P_X$ individually. This will require basic techniques and concepts from real algebraic geometry, which we  condense below. The reader familiar with real algebraic geometry may jump directly to Section \ref{sec:thedimensionofP}. 

\subsection{Real Algebraic Geometry Preliminaries}
\label{sec:basicalgebraicgeometry}

We will need some elementary facts about algebraic and semialgebraic sets, as well as appropriate notions of dimension for each of these. A thorough introduction can be found, for instance, in Sections 2 and 3 of \cite{coste2000introduction}.

An \textit{algebraic set} (or, more formally, a \textit{real affine algebraic set}) is a subset of $\R^n$ defined as the zero set of a family of polynomials with real coefficients. It is easy to see from the definition that any finite union or finite intersection of algebraic sets is still an algebraic set. Similarly, a \textit{semialgebraic set} is a subset of $\bbR^n$ defined by a family of polynomial inequalities. Any algebraic set is semialgebraic, but the reverse need not be true.

An algebraic set $\calX$ is \textit{irreducible} if it cannot be expressed as a disjoint union of two algebraic sets strictly contained in $\calX$. It well known \cite[Section 2.8]{bochnak2013real} that  any algebraic set $\calX$ admits a unique decomposition of the form $\calX =\bigcup_{i=1}^k \calX_i$ where each $\calX_i$ is an irreducible algebraic set and such that for no $i\neq j$ is $\calX_i$ contained in $\calX_j$. If $\calX$ is an irreducible algebraic set we define the \textit{algebraic dimension} of $\calX$, denoted $\dim \calX$, as the maximum integer $d$ such that there exists a chain of the form $\calX_0 \subset \calX_1 \subset \cdots \subset \calX_d =\calX$, where each $\calX_i$ is an irreducible algebraic set and each containment is strict. If $\calX$ is any algebraic set and $\calX =\bigcup_{i=1}^k \calX_i$ is its decomposition into irreducible sets, we define the algebraic dimension of $\calX$ as $\dim \calX =\max_{i\in [k]} \dim \calX_i$. It follows from these definitions that if $\calX$ and $\calY$ are two  algebraic  sets, with $\calX$ irreducible, and $\calX$ is not contained in $\calY$, then $\dim \calX\cap \calY < \dim \calX$. 

The notion of algebraic dimension for algebraic sets can be extended to a notion of dimension for semialgebraic sets via the \textit{cylindrical algebraic decomposition}, which we describe here. Any semialgebraic set admits a decomposition of the form $\calS = \bigcup_{i=1}^k \calC_i$, where the $\calC_i$ are disjoint semialgebraic subsets dieffeomorphic to the open hypercube $(0, 1)^{d_i}$ for some  nonnegative integer $d_i$ \cite[Corollary 3.8]{coste2000introduction}. With this setup we define the dimension of $\calS$ as $\dim \calS =\max_{i\in [k]} d_i$.  We also remind the reader that the Hausdorff dimension of a semialgebraic set coincides with the notion of dimension described here. 
 
Finally, we will use a fundamental fact about projections of semialgebraic sets. If $\calS\subset \R^{n+1}$ is  semialgebraic  and $\Pi: \R^{n+1}\to \R^n$ is the projection onto the first $n$ coordinates, then $\Pi \calS$ is a semialgebraic subset of $\R^n$ \cite[Theorem 2.3]{coste2000introduction}, and moreover $\dim \Pi (\calS) \leq \dim \calS$ \cite[Lemma 3.17]{coste2000introduction}. 
\subsection{The Dimension of $\P$ }
\label{sec:thedimensionofP}
We begin by proving the main technical result of this section. 

\begin{proposition}
\label{prop:codimension}
For any $X\in \calA(G)$, $\P_X$ is a semialgebraic set of dimension at most $m-\cc X+1$.  
\end{proposition}
\begin{proof}
Let $p=\cc X$ and $T_1, \dots,T_p$ be the trees induced by $X$.  For any Hermitian edge weights $a: E \to \bbC$, let $x =\Re(a)$ and $y=\Im(a)$, that is, for every $e\in E$ we write  $a_e= x_e + i y_e$, with $x_e, y_e \in \bbR$. View the characteristic polynomials of the $T_i$ as polynomials in the $x_e, y_e, b_v$ and $z$, namely, define $P_i(x, y, b, z)= \det(z-A_{T_i})$. We will first show that each $P_i(x, y, b, z)$ is a polynomial with real coefficients. Remember that $ \Im P_i(x, y, b, z)$  is a polynomial with real coefficients in the aformentioned variables. Now, since $A_{T_i}$ is Hermitian, for any choice of $x, y\in \R^{|E|/2}$ and $b \in \R^{|V|}$, we have that $\det(z-A_{T_i})\in \R$, so $\Im P_i \equiv 0$  on $\R^{m+1}$ and hence $\Im P_i$ is the zero polynomial. It then follows that $P_i = \Re P_i$, which means that $P_i\in \mathbb{R}[x, y, b, z]$.  

Now, for $i, r\in [p]$ define the  algebraic sets
$$
    \mathcal{X}_i=\{(x, y, b, z) \in \R^{m+1}: P_i(x, y, b, z)=0 \}\qquad \text{and} \qquad \mathcal{X}_{\leq r} =\bigcap_{i=1}^r \mathcal{X}_i.
$$
We will now show that $\mathcal{X}_{\leq r}$ has codimension at least $r$ for all $r\in [p]$, which implies in particular that that $\mathcal{X}_{\leq p}$ has algebraic dimension at most $m+1-p$. For the base case note that since $P_1$ is not the zero polynomial, $\mathcal{X}_1$ is a proper algebraic subset of $\R^{m+1}$, and since $\R^{m+1}$ is irreducible $\dim (\mathcal{X}_1) < \dim(\R^{m+1}) = m+1$.  Now, for $1\leq r\leq p-1$ assume that $\dim(\mathcal{X}_{\leq r})\leq m-r+1$,  define the subspace 
$$
    W_r = \left\{(x, y, b, z)\in \R^{m+1}: b_v=0\text{ for } v\in V(G)\setminus \bigcup_{i=1}^r V(T_i)\right\},
$$
and let $\mathcal{Y} = \calX_{\leq r} \cap W_r$. 

The set $\mathcal{Y}$ is itself algebraic. Moreover, since for each $i\leq r$ the set $P_i$ depends only on those $b_v$'s with $v\in \bigcup_{i=1}^r V(T_i)$, we have that $\mathcal{X}_{\leq r} = \mathcal{Y} \times W_r^\perp$. Let $\bigcup_{i=1}^{n_1} \mathcal{Y}_i$ be the decomposition of $\mathcal{Y}$ into irreducible components. Since $\mathcal{Y}_i$ and $W_r^\perp$ are both irreducible $\mathcal{Y}_i\times W_r^\perp$ is as well, and hence $\bigcup_{i\in [n_1]} (\mathcal{Y}_i\times W_r^\perp)$ is in fact the decomposition of $\mathcal{X}_{\leq r}$ into irreducible components. On the other hand
$$
    \mathcal{X}_{\leq r+1} = \mathcal{X}_{\leq r} \cap \mathcal{X}_{r+1} = \bigcup_{i=1}^{n_1}(\mathcal{Y}_i\times W_r^\perp)\cap \calX_{r+1}.
$$
We will now show that, for every $i\in [n_1]$,   $\mathcal{Y}_i\times W_r^\perp$ is not contained in $\calX_{r+1}$. 

Indeed, fix $(x, y, b, z)\in \mathcal{Y}_i\times W_r^\perp$. Adding a constant $c\in \R$ to the $b_v$ with $v\in V(T_{r+1})$ has the effect of shifting the spectrum of $A_{T_{r+1}}$ by $c$. We can then find $b'\in \R^{|V|}$ with the property that $b_v'=b_v$ for all $v\in V(G)\setminus V(T_{r+1})$ and such that $P_{r+1}(x, y, b', z)\neq 0$. By construction we have $(x, y, b', z)\in \mathcal{Y}_i\cap W_r^\perp $ and $(x, y, b',z)\notin \mathcal{X}_{r+1}$ as we wanted to show.   This implies that $(\mathcal{Y}_i\times W_r^\perp)\cap \calX_{r+1}$ is a proper subset of $\mathcal{Y}_i\times W_r^\perp$, and since the latter is irreducible we get
$$
    \dim \left((\mathcal{Y}_i\times W_r^\perp)\cap \mathcal{X}_{r+1}\right) < \dim( \mathcal{Y}_i\times W_r^\perp)\leq \dim(\mathcal{X}_{\leq r})\leq m+1-r,
$$
concluding the inductive step. Finally, let $\Pi: \R^{m+1}\to \R^m$ be the projection defined by $\Pi(x, y, b, z) = (x, y, b)$, and note that $\Pi \mathcal{X}_{\leq p} = \mathcal{P}_X$. From the results mentioned in Section \ref{sec:basicalgebraicgeometry}, $\Pi \mathcal{X}_{\leq p}$ is a semialgebraic set whose dimension is less or equal to that of $\mathcal{X}_{\leq p}$, and in turn $\dim \mathcal{X}_{\leq p}\leq m-p+1$.
\end{proof}

We are now ready to prove Theorem \ref{thm:setofparameterswithpp}. 

\begin{proof}[Proof of Theorem \ref{thm:setofparameterswithpp}.]
By Proposition \ref{prop:codimension} and Equation \eqref{eq:pointdecomp}, $\P$ is semialgebraic with
$$
    \mathrm{codim}\,\P \ge \min_{X\in \mathcal{A}(G)} \cc X-1 \geq \min_{X\in \mathcal{A}(G)} \partial X, 
$$
and we want to further lower bound the latter quantity. As $G$ has at least one cycle, $\partial X\neq \emptyset$ for all $X\in \calA (G)$, so  $\P$ has codimension at least 1. Now, if $d_{\min} \geq 2$ take any tree $T$ induced by $X$. Any vertex $v$ of $T$ must be connected to at least $d_{\min}-1$ distinct vertices in $\partial X$, and hence $\cc X -1\geq \partial X\geq d_{\min}-1$. This proves the bound $\mathrm{codim}\, \P\geq \max\{d_{\min}-1, 1\}$.  

We show finally that $\calP^c$ is open. For every $X\in \calA(G)$ denote the forest induced by $X$ by $\F_X$. Fix $(a,b) \in \P^c$. By Theorem \ref{thm:aomoto-converse}, for every $X\in \calA(G)$ the Jacobi matrices of the trees in $\F_X$ (with weights and potentials given by $a$ and $b$), do not have a common eigenvalue. Now, define 
$$
    S = \bigcup_{X\in \calA(G)} \bigcup_{T\in \F_X} \Spec \jacobi{T}.
$$
As $\calF_X$ is finite for each of the finitely many $X \in \calA(G)$, we may safely define $\Delta > 0$ to be the smallest distance between two distinct points in $S$. We will show that if $(a', b')\in \R^m$ satisfies $\|(a, b)-(a', b')\|_{2} < \Delta/2$ then $(a', b')\in \calP^c$. 

Assume otherwise. Then there exists an $X\in \calA(G)$ such that the Jacobi matrices with parameters in $(a', b')$ of the trees in $\F_X$  have a common eigenvalue $\lambda$. Let $T_1, \dots, T_p$ be the trees in $\F_X$ with parameters in $(a, b)$ and let $T_1', \dots, T_p'$ denote the same trees but with parameters in $(a', b')$.  For every $i$ let $\lambda_i$ be the closest point in $ \Spec A_{T_i}$ to $\lambda$. Since $(a, b)\in \calP^c$ we have  $\lambda_i\neq \lambda_j$ for some $i,j$. On the other hand since $\|A_{T_i}-A_{T_i'}\|\leq \|A_{T_i}-A_{T_i'}\|_{F} \leq \|(a, b)\|_2 <\Delta/2$ and similarly $\|A_{T_j}-A_{T_j'}\| <\Delta/2$, the triangle inequality and Weyl's inequality together imply
$$
    |\lambda_i - \lambda_j| \le |\lambda_i - \lambda | + |\lambda_j - \lambda| < \Delta/2 + \Delta/2 = \Delta,
$$
contradicting the definition of $\Delta$.
\end{proof}